\documentclass[11pt]{amsart}%
\usepackage{amsfonts}
\usepackage{amsmath}
\usepackage{amssymb}
\usepackage{graphicx}
\usepackage{geometry}
\usepackage{verbatim}
\usepackage{hyperref}
\usepackage{charter,eulervm}
\providecommand{\U}[1]{\protect\rule{.1in}{.1in}}

\swapnumbers
\theoremstyle{plain}
\newtheorem{theorem}{Theorem}[subsection]
\newtheorem{corollary}[theorem]{Corollary}
\newtheorem{proposition}[theorem]{Proposition}

\newtheorem{lemma}[theorem]{Lemma}

\theoremstyle{definition}
\newtheorem{definition}[theorem]{Definition}

\newtheorem{example}[theorem]{Example}

\theoremstyle{remark}

\numberwithin{equation}{subsection}

\DeclareMathOperator{\coz}{coz}
\DeclareMathOperator{\cone}{con}
\begin{document}
\title[Truncs]{Truncated abelian lattice-ordered groups II:\\the pointfree (Madden) representation}
\author{Richard N. Ball}
\address{Department of Mathematics, University of Denver, Denver, CO 80208, U.S.A.}
\thanks{File name: \texttt{NewTruncFrame.tex}. }
\date{28 June 2014}

\begin{abstract}
This is the second of three articles on the topic of truncation as an
operation on divisible abelian lattice-ordered groups, or simply $\ell
$-groups. This article uses the notation and terminology of the first article
and assumes its results. In particular, we refer to an $\ell$-group with
truncation as a truncated $\ell$-group, or simply a trunc, and denote the
category of truncs with truncation morphisms by $\mathbf{AT}$.

Here we develop the analog for $\mathbf{AT}$ of Madden's pointfree
representation for $\mathbf{W}$, the category of archimedean $\ell$-groups
with designated order unit. More explicitly, for every archimedean trunc $A$
there is a regular Lindel\"{o}f frame $L$ equipped with a designated point
$\ast:L\rightarrow2$, a subtrunc $\widehat{A}$ of $\mathcal{R}_{0}L$, the
trunc of pointed frame maps $\mathcal{O}_{0}\mathbb{R}\rightarrow L$, and a
trunc isomorphism $A\rightarrow\widehat{A}$. A pointed frame map is just a
frame map between frames which commutes with their designated points, and
$\mathcal{O}_{0}\mathbb{R}$ stands for the pointed frame which is the topology
$\mathcal{O}\mathbb{R}$ of the real numbers equipped with the frame map of the
insertion $0 \to \mathbb{R}$. $\left(  L,\ast\right)  $ is
unique up to pointed frame isomorphism with respect to its properties.
Finally, we reprove an important result from the first article, namely that
$\mathbf{W}$ is a non-full monoreflective subcategory of  $\mathbf{AT}$.

\end{abstract}
\maketitle
\tableofcontents

\section{Introduction}

\subsection{A brief synopsis}

We develop the analog for truncs of Madden's pointfree representation for
$\mathbf{W}$, the category of archimedean $\ell$-groups with order unit
(\cite{Madden:1990}, \cite{MaddenVermeer:1990}, \cite{BallHager:1991}). We
begin by showing in Section \ref{Sec:4} that, for an arbitrary trunc $A$, the
truncation kernels form a regular Lindel\"{o}f frame $\mathcal{K}A$. In
Section \ref{Sec:5} we then provide a subtrunc $\underline{A}$ of
$\mathcal{RK}A$ and a trunc isomorphism $\kappa_{A}:A\rightarrow\underline{A}%
$. This much directly parallels Madden's development in $\mathbf{W}$. But here
an obstacle rears its head, an obstacle which is invisible in $\mathbf{W}$.

Although the representation $\kappa_{A}:A\rightarrow\underline{A}%
\leq\mathcal{RK}A$ is faithful, i.e., one-one, it is not functorial, as we
show by example. However, we restore functoriality by the simple stratagem of
attaching a designated point to each frame under consideration and restricting
our attention to the frame maps which respect the points by commuting with
them. This requires a systematic study of the category of pointed frames, and
we carry this out in Section \ref{Sec:1}. The corresponding algebraic
construct is $\mathcal{R}_{0}L$, the trunc of pointed frame maps
$\mathcal{O}_{0}\mathbb{R}\rightarrow L$, where $\mathcal{O}_{0}\mathbb{R}$
designates the pointed frame of the real numbers, the point being the frame
map of the insertion $0 \to \mathbb{R}$. (For instance, if $L$
is compact and regular then $L=OX$ for some compact Hausdorff space $X$ with
designated point $x_{0}$, and $\mathcal{R}_{0}L$ is isomorphic to
$\mathcal{D}_{0}X$, the trunc of continuous almost-finite extended-real valued
functions on $X$ which vanish at $x_{0}$.) The section culminates in the
development of the spectrum $\mathcal{M}A$ of the trunc $A$; it is the pointed
frame obtained by formally adjoining a point to the frame $\mathcal{K}A$ of
truncation kernels.

All that has gone before is preparation for Section \ref{Sec:2}. There we work
out the desired representation $A\rightarrow\widehat{A}\leq\mathcal{R}%
_{0}\mathcal{M}A$, and show it to be faithful and functorial. We conclude the
paper with Section \ref{Sec:3}, in which we show that the unital objects form 
a non-full monoreflective subcategory of the category of truncs,
thereby confirming the proof of the same fact in \cite{Ball:2013}.

For background on rings of continuous functions, we direct the reader to
Gillman and Jerison's masterpiece \cite{GillmanJerison:1960}, and for a general
reference on $\ell$-groups, to Darnel's fine text \cite{Darnel:1995}. 

\subsection{The basic definitions}

For the purposes of this article, we define truncation as follows.

\begin{definition}
\label{Def:3}A \emph{truncation } on an $\ell$-group $A$ is a unary operation
$A^{+}\rightarrow A^{+}$, written $a\longmapsto\overline{a}$, which has the
following properties for all $a,b\in A^{+}$.

\begin{enumerate}
\item[($\mathfrak{T}1$)] $a\wedge\overline{b}\leq\overline{a}\leq a$

\item[($\mathfrak{T}2$)] If $\overline{a}=0$ then $a=0$.

\item[($\mathfrak{T}3$)] If $na=\overline{na}$ for all $n$ then $a=0$.

\item[($\mathfrak{T}4$)] $\bigcap_{\mathbb{N}}a\ominus n^{\bot\bot}=0$ for all
$a\in A^{+}$. The symbol $a\ominus n$ stands for $n\left(  \left(  a/n\right)
\ominus1\right)  =a-n\overline{a/n}$. This axiom can therefore be formulated as

\item[($\mathfrak{T}4^{\prime}$)] $\forall~b>0~\exists~c~\exists~n~\left(
0<c\leq b\text{ and }c\wedge a\ominus n=0\right)  .~$
\end{enumerate}

\noindent\noindent An $\ell$-group equipped with a truncation is called a
\emph{truncated }$\ell$\emph{-group}, or \emph{trunc} for short. A
\emph{truncation morphism} is an $\ell$-homomorphism $f:A\rightarrow B$
between truncs which preserves the truncation, i.e., $f\left(  \overline
{a}\right)  =\overline{f\left(  a\right)  }$ for all $a\in A^{+}$. We denote
the category of truncs with truncation morphisms by $\mathbf{AT}$.
\end{definition}

The definition of truncation in \cite{Ball:2013} requires only the first three
axioms; the fourth appears in that article as the requirement that each
element of $A^{+}$ be archimedean, and it follows from \cite[5.3.1]{Ball:2013}
that $A$ itself is archimedean. In summary, we get the following.

\begin{proposition}
The following are equivalent for an $\ell$-group $A$.

\begin{enumerate}
\item $A$ is truncated in the sense of Definition \ref{Def:3}.

\item $A$ is archimedean and truncated in the sense of \cite{Ball:2013}.

\item $A$ is archimedean and satisfies ($\mathfrak{T}1$) and ($\mathfrak{T}2$).
\end{enumerate}
\end{proposition}

\begin{proof}
The discussion prior to the proposition establishes the equivalence of the
first two conditions. The equivalence of the second and third is just the
observation that an archimedean trunc has no infinitesimals and therefore
satisfies ($\mathfrak{T}3$). \cite[5.3.1]{Ball:2013}.
\end{proof}

We reiterate two points for emphasis.

\begin{itemize}
\item \emph{The truncs considered in this article are truncs in the sense of
\cite{Ball:2013}.} As a result, all the descriptive results from
\cite{Ball:2013}, including the identities in Section 3.3, apply to the more
specialized truncs considered here.

\item \emph{The truncs considered in this article are necessarily
archimedean.}
\end{itemize}

\begin{definition}
A trunc $A$ is said to be \emph{unital with unit }$u$ if $\overline{a}=a\wedge
u$ for all $a\in A^{+}$. We denote by $\mathbf{W}$ the full subcategory of
$\mathbf{AT}$ comprised of the unital truncs.
\end{definition}

\section{Truncation ideals\label{Sec:4}}

We have set before us the task of developing a representation theory for
$\mathbf{AT}$ generalizing the pointfree Madden representation for
$\mathbf{W}$. The universal objects for the latter are of the form
$\mathcal{R}M$, the $\mathbf{W}$-object of frame maps $\mathcal{O}%
\mathbb{R}\rightarrow M$ for some frame $M$. In fact, an arbitrary
$\mathbf{W}$-object $A$ is captured as a subobject of $\mathcal{RM}A$, where
$\mathcal{M}A$, the \emph{Madden frame of }$A$, is the frame of $\mathbf{W}%
$-kernels of $A$. It is therefore to be expected that the frame of kernels of
the morphisms of $\mathbf{AT}$ plays a central role in the representation of
truncs. We will refer to such kernels as \emph{truncation kernels}. Our first
task is to develop criteria which will enable us to recognize them.

\subsection{An internal characterization}

An \emph{archimedean kernel} of an $\ell$-group $A$ is the set of elements
sent to zero by an $\ell$-homomorphism into an archimedean codomain. Of the
several known characterizations of such kernels, we will use only the
simplest, which we reprove here in the interests of a self-contained
treatment. We continue to assume that $A$ represents an arbitrary trunc.

\begin{lemma}
\label{Lem:4}A subset $K$ of a trunc $A$ is an archimedean kernel iff it is a
convex $\ell$-subgroup such that, for all $a\in A^{+}$,
\[
\left(  \exists~c\in A^{+}~\forall~n\in\mathbb{N}~\left(  na-c\right)  ^{+}\in
K\right)  \Longrightarrow a\in K.
\]

\end{lemma}

\begin{proof}
The quotient group $A/K$ can be endowed with an order making the quotient map
an $\ell$-homomorphism iff $K$ is a convex $\ell$-subgroup. We must show that
the archimedean property of $A/K$ is equivalent to the condition displayed
above. This follows directly from the fact that, for $a,c\in A^{+}$ and
$n\in\mathbb{N}$,
\begin{gather*}
\left(  na-c\right)  ^{+}\in K\Longleftrightarrow K=K+\left(  na-c\right)
\vee0\Longleftrightarrow K+c=\left(  K+na\right)  \vee\left(  K+c\right)
\Longleftrightarrow\\
K+na\leq K+c\Longleftrightarrow n\left(  K+a\right)  \leq K+c. \qedhere
\end{gather*}

\end{proof}

We shall say that a subset $K$ of a trunc $A$ is \emph{absorbing} if
$\overline{a}\in K$ implies $a\in K$ for all $a\in A^{+}$.

\begin{lemma}
\label{Lem:5}A subset $K\subseteq A$ is a truncation kernel iff it is an
absorbing archimedean kernel.
\end{lemma}

\begin{proof}
We have already mentioned that truncs are archimedean, hence a truncation
kernel must be an archimedean kernel. And the absorbing property of a
truncation kernel $K$ is clearly necessitated by the fact that $A/K$ must
satisfy truncation axiom ($\mathfrak{T}2$). On the other hand, if $K$ is an
absorbing archimedean kernel then define
\[
\overline{K+a}=K+\overline{a},\;a\in A^{+}.
\]
This is the only definition of truncation on $A/K$ which makes the quotient
map into a truncation morphism. Moreover, the truncation is well-defined just
because $\left\vert \overline{a}-\overline{b}\right\vert \leq\left\vert
a-b\right\vert $, $a,b\in A^{+}$, by \cite[Lemma 3.3.1(4)]{Ball:2013}. It is
then straightforward to verify that the operation thus defined satisfies the
truncation axioms.
\end{proof}

\begin{corollary}
The truncation kernels of a $\mathbf{W}$-object coincide with its $\mathbf{W}
$-kernels.
\end{corollary}

\begin{proof}
A $\mathbf{W}$-kernel is just an archimedean kernel $K$ such that $a\wedge1\in
K$ forces $a\in K$ for all $a\in A^{+}$. But the truncation in a $\mathbf{W}%
$-object is taken to be $\overline{a}=a\wedge1$, so the latter condition is
exactly the absorbing property of $K$.
\end{proof}

\subsection{An external characterization}

We use $\mathcal{K}A$ to designate the family of truncation ideals of $A$,
considered as a lattice in the inclusion order. Our task for the remainder of
the section is to show that $\mathcal{K}A$ is a regular Lindel\"{o}f frame;
such a frame is completely regular. This we do in Theorem \ref{Thm:1}.

First, observe that $\mathcal{K}A$ is closed under intersection, a consequence
of the fact that $\mathbf{AT}$ is evidently closed under products. This leads
to the notion of the truncation kernel generated by a subset $B\subseteq A$,
which we denote%
\[
\left[  B\right]  \equiv\bigcap\left\{  K:\mathcal{K}A\ni K\supseteq
B\right\}  .
\]

It is helpful to have an external description of $\left[  B\right]  $, and
such a description requires a little terminology. Every ordinal number
$\alpha$ can be expressed in the form $\alpha=\beta+k$ for a unique finite
ordinal $k$ and limit ordinal $\beta$. (We take $\beta=0$ to be a limit
ordinal.). We will say that $\alpha$\emph{ is even or odd} depending on
whether $k$ is even or odd.

For a subset $B\subseteq A$, let $\left\langle B\right\rangle $ designate the
convex $\ell$-subgroup generated by $B$. Now define%
\begin{align*}
B^{0}  &  \equiv\left\langle B\right\rangle ,\\
B^{\alpha+1}  &  \equiv\left\langle a\in A^{+}:\exists~c\in A^{+}~\forall
~n\in\mathbb{N}~\left(  n\left\vert a\right\vert -c\right)  ^{+}\in B^{\alpha
}\right\rangle ,\;\alpha\text{ even,}\\
B^{\alpha+1}  &  \equiv\left\langle a\in A^{+}:\overline{\left\vert
a\right\vert }\in B^{\alpha}\right\rangle ,\;\alpha\text{ odd,}\\
B^{\beta}  &  \equiv\bigcup_{\alpha<\beta}B^{\alpha}\text{, }\beta\text{ a
limit ordinal,}\\
B^{\infty}  &  \equiv B^{\alpha}\text{ for some (any) }\alpha\text{ such that
}B^{\alpha}=B^{\alpha+1}\text{.}%
\end{align*}

\begin{lemma}
\label{Lem:24}$\left[  B\right]  =B^{\infty}=B^{\omega_{1}}$ for any subset
$B\subseteq A$.
\end{lemma}

\begin{proof}
This follows directly from Lemma \ref{Lem:5}.
\end{proof}

Polars are good examples of truncation kernels.

\begin{corollary}
\label{Cor:7}For any subset $B\subseteq A$, the polar
\[
B^{\bot}\equiv\left\{  a:\left\vert a\right\vert \wedge\left\vert b\right\vert
=0\text{ for all }b\in B\right\}
\]
generated by $B$ is a truncation kernel.
\end{corollary}

\begin{proof}
A straightforward induction reveals that $B^{\bot}=B^{\bot\alpha}$ for all
$\alpha$.
\end{proof}

\begin{corollary}
\label{Cor:2}For any $a\in A^{+}$, $\left[  a\right]  ^{\ast}=a^{\bot}$.
Consequently, $\left[  a\right]  ^{\ast\ast}=a^{\bot\bot}$, and it follows
that $a\in\overline{a}^{\bot\bot}$, i.e., $a^{\bot\bot}=\overline{a}^{\bot
\bot}$.
\end{corollary}

\begin{proof}
Since $a^{\bot}$ is the pseudocomplement of $\left\langle a\right\rangle $ in
the lattice $\mathcal{L}A$ of convex $\ell$-subgroups of $A$, the fact that it
is a truncation kernel implies that $a^{\bot}$ must be the pseudocomplement of
$\left[  a\right]  $ in the lattice $\mathcal{K}A$ of truncation kernels of
$A$. And we have $a\in\left[  a\right]  = \left[  \overline{a}\right]
\leq\left[  \overline{a}\right]  ^{\ast\ast}= \overline{a}^{\bot\bot}.$
\end{proof}

\begin{lemma}
\label{Lem:13}For convex $\ell$-subgroups $A_{i}\subseteq A$, $\left[
A_{1}\right]  \cap\left[  A_{2}\right]  =\left[  A_{1}\cap A_{2}\right]  $.
\end{lemma}

\begin{proof}
Clearly $\left[  A_{1}\right]  \cap\left[  A_{2}\right]  \supseteq\left[
A_{1}\cap A_{2}\right]  $, and a simple induction can be used to show that
$A_{1}^{\alpha}\cap A_{2}^{\alpha}\subseteq\left(  A_{1}\cap A_{2}\right)
^{\alpha}$ for all $\alpha$. In view of Lemma \ref{Lem:24}, this assertion for
$\alpha=\omega_{1}$ establishes that $\left[  A_{1}\right]  \cap\left[
A_{2}\right]  \subseteq\left[  A_{1}\cap A_{2}\right]  $.
\end{proof}

The reader should be warned that Lemma \ref{Lem:13} is false without the
hypothesis of convexity.

\begin{lemma}
\label{Lem:7}An element $b\in A^{+}$ lies in a truncation kernel $K$ iff
$n\overline{b/n}\in K$ for some $n\in\mathbb{N}$ iff $n\overline{b/n}\in K$
for all $n\in\mathbb{N}$
\end{lemma}

\begin{proof}
If $K$ is any convex $\ell$-subgroup of $A$ then, for any $n\in\mathbb{N}$,
$b\in K$ implies $\overline{b/n}\in K$ by convexity since $b\geq
b/n\geq\overline{b/n}\geq0$, with the result that $n\overline{b/n}\in K$ for
all $n$. On the other hand, if $K$ is a truncation kernel and $n\overline
{b/n}\in K $ for some $n\in\mathbb{N}$ then $\overline{b/n}\in K$, which
implies $b/n\in K$ by Lemma \ref{Lem:5}, so that $b\in K$.
\end{proof}

\begin{lemma}
\label{Lem:1}For any $a\in A^{+}$,%
\[
\left[  na\ominus1:n\in\mathbb{N}\right]  =\left[  a\ominus\frac{1}{n}%
:n\in\mathbb{N}\right]  =\left[  a\ominus0\right]  =\left[  a\right]  .
\]

\end{lemma}

\begin{proof}
Let $K\equiv\left[  na\ominus1:n\in\mathbb{N}\right]  $. We have
\[
0\leq na\ominus1=n\left(  a\ominus\left(  1/n\right)  \right)  \leq n\left(
a\ominus0\right)  =na,
\]
so that $K\subseteq\left[  a\right]  $. But since $K=K+na\ominus1=K+\left(
na-\overline{na}\right)  $, which is to say that $K+na=K+\overline{na}$ for
all $n$, we also have $a\in K$ because $A/K$ must satisfy ($\mathfrak{T}3$).
\end{proof}

\begin{corollary}
\label{Cor:4}For any $a\in A^{+}$ and $r\in\mathbb{Q}^{+}$,%
\[
\bigvee_{s>r}\left[  a\ominus s\right]  =\left[  a\ominus\left(  r+\frac{1}%
{n}\right)  :n\in\mathbb{N}\right]  =\left[  a\ominus r\right]  .
\]

\end{corollary}

\begin{proof}
Using Lemma 3.3.6 from \cite{Ball:2013}, express $a\ominus\left(
r+1/n\right)  $ as $a\ominus r\ominus\frac{1}{n}$, and apply Lemma \ref{Lem:1}.
\end{proof}

\subsection{The frame of truncation kernels}

\begin{proposition}
\label{Prop:3}$\mathcal{K}A$ is a frame, and the frame operations are
\[
K_{1}\wedge K_{2}=K_{1}\cap K_{2}\;\text{and \ }\bigvee_{I}K_{i}=\left[
K_{i}:i\in I\right]  ,\;\left\{  K_{i}:i\in I\right\}  \subseteq\mathcal{K}A.
\]
The pseudocomplemented elements of $\mathcal{K}A$ are the polars, i.e.,
$K^{\ast}=K^{\bot}$ for $K\in\mathcal{K}A$. Moreover, for $a_{i}\in A^{+}$,
\[
\left[  a_{1}\right]  \cap\left[  a_{2}\right]  =\left[  a_{1}\wedge
a_{2}\right]  \;\text{and\ }\;\left[  a_{1}\right]  \cup\left[  a_{2}\right]
=\left[  a_{1}\vee a_{2}\right]  .
\]

\end{proposition}

\begin{proof}
Binary meets distribute across arbitrary joins in $\mathcal{K}A$ because the
same is true in the frame $\mathcal{L}A$ of convex $\ell$-subgroups of $A$,
and the linkage between the two is provided by Lemma \ref{Lem:13}. In detail,
for truncation kernels $K$, $K_{i}$, $i\in I$, we have
\[
K\wedge\bigvee_{I}K_{i}=K\cap\left[  \bigsqcup_{I}K_{i}\right]  =\left[
K\cap\bigsqcup_{I}K_{i}\right]  =\left[  \bigsqcup_{I}\left(  K\cap
K_{i}\right)  \right]  =\bigvee_{I}\left(  K\wedge K_{i}\right)  ,
\]
where $\bigsqcup$ signifies the join in $\mathcal{L}A$. Similarly, the
displayed equations governing principle truncation kernels hold just because
their analogs hold in $\mathcal{L}A$, i.e., $\left\langle a_{1}\right\rangle
\wedge\left\langle a_{2}\right\rangle =\left\langle a_{1}\wedge a_{2}%
\right\rangle $ and $\left\langle a_{1}\right\rangle \vee\left\langle
a_{2}\right\rangle =\left\langle a_{1}\vee a_{2}\right\rangle $. Finally, from
Corollary \ref{Cor:2} we learn that the pseudocomplemented elements of
$\mathcal{L}A$, namely the polars, are present in $\mathcal{K}A$ and therefore
constitute the pseudocomplemented elements of $\mathcal{K}A$.
\end{proof}

\begin{lemma}
\label{Lem:9}For any $a\in A^{+}$ we have $\left[  a\ominus r\right]
\prec\left[  a\ominus s\right]  $ for $s<r$ in $\mathbb{Q}^{+}$.
\end{lemma}

\begin{proof}
It is sufficient to establish this for $s=0$, for the general case follows by
expressing $a\ominus r$ as $\left(  a\ominus s\right)  \ominus\left(
r-s\right)  $ and applying the special case. We want to show that $a\ominus
r^{\bot}\vee\left[  a\right]  =A=\top$. But according to \cite[3.3.7]%
{Ball:2013}, for each $b\in A^{+}$ there is some $x\in a\ominus r^{\bot}$ for
which $2\left(  a\vee x\right)  \geq r\overline{b/r}$. It follows that
$\overline{b/r}$ lies in $\left\langle a\ominus r^{\bot},a\right\rangle $, the
convex $\ell$-subgroup generated by $a\ominus r^{\bot}\cup\left\{  a\right\}
$. And since the truncation kernel $a\ominus r^{\bot}\vee\left[  a\right]  $
generated by $a\ominus r^{\bot}\cup\left\{  a\right\}  $ is absorbing, it must
contain $b$. All this is to say that $a\ominus r^{\bot}\vee\left[  a\right]
=A=\top$.
\end{proof}

\begin{lemma}
\label{Lem:2}$\mathcal{K}A$ is regular.
\end{lemma}

\begin{proof}
By Lemmas \ref{Lem:1} and \ref{Lem:9}, each principal truncation kernel
$\left[  a\right]  $ can be expressed as a countable join of elements rather
below it, viz.\ $\left[  a\right]  =\bigvee_{\mathbb{N}}\left[  a\ominus
\left(  1/n\right)  \right]  $. Since the principal truncation ideals generate
$\mathcal{K}A$, we can be sure it is regular.
\end{proof}

We write $S_{0}\subseteq_{\omega_{1}}S$ to mean that $S_{0}$ is a subset of
$S$ which is at most countable.

\begin{lemma}
\label{Lem:3}For any subset $S\subseteq\mathcal{K}A$,
\[
\bigvee S=\bigcup\left\{  \bigvee S_{0}:S_{0}\subseteq_{\omega_{1}}S\right\}
.
\]

\end{lemma}

\begin{proof}
By Lemma \ref{Lem:5} it is sufficient to show that the set displayed on the
right, call it $K$, is an absorbing archimedean kernel. That $K$ is absorbing
is clear, so consider elements $a,c\in A^{+}$ such that $\left(  na-c\right)
^{+}\in K$ for all $n$, say $\left(  na-c\right)  ^{+}\in\bigvee S_{n}$ for
some $S_{n}\subseteq_{\omega_{1}}S$. Put $S_{0}\equiv\bigcup_{n}S_{n}%
\subseteq_{\omega_{1}}S$ and $L\equiv\bigvee S_{0}\subseteq K$. Then $L$ is a
truncation kernel containing $\left(  na-c\right)  ^{+}$ for all $n$, hence
$a\in L$ by Lemma \ref{Lem:4}, hence $K$ is an archimedean kernel by a second
application of the same lemma.
\end{proof}

We summarize the results of our investigation to this point..

\begin{theorem}
\label{Thm:1}$\mathcal{K}A$ is a regular Lindel\"{o}f frame.
\end{theorem}

\begin{proof}
We showed $\mathcal{K}A$ to be regular in Lemma \ref{Lem:2}; the fact that it
is Lindel\"{o}f is an immediate consequence of Lemma \ref{Lem:3}.
\end{proof}

\subsection{More about truncation kernels}

The functorial representation we seek rests on the properties of certain
truncation kernels of $A$. We develop these properties in this section. This
material, however, will not be relevant until the main Section \ref{Sec:2}.
The reader may therefore skip this section without loss of continuity, and
refer to these results only as they are put to use.

Two truncation kernels associated with a given $a\in A^{+}$ play prominent
roles in what follows. We define%
\[
a\blacktriangleright r\equiv\left[  a\ominus r\right]  \;\text{and\ }%
a\blacktriangleleft r\equiv\bigvee_{0<s<r}a\ominus s^{\bot},\;r\in
\mathbb{Q}^{+}.
\]
The join in the definition of $a\blacktriangleleft r$ is computed in
$\mathcal{K}A$, and can be handily expressed as
\[
\bigvee_{s<r}a\blacktriangleright s^{\ast}\;\text{or\ }%
{\textstyle\bigvee_{\mathbb{N}}}
a\ominus\left(  r-\frac{1}{n}\right)  ^{\bot}\;\text{or\ }\left[
{\textstyle\bigcup_{\mathbb{N}}}
a\ominus\left(  r-\frac{1}{n}\right)  ^{\bot}\right]  .
\]
We record the basic properties of these kernels in Lemma \ref{Lem:22}, whose
proof requires a preliminary observation.

\begin{lemma}
\label{Lem:12}Suppose that $b_{1}\succ b_{2}$ and $c_{1}\succ c_{2}$ in some
frame. Then%
\[
b_{1}^{\ast}\vee c_{1}^{\ast}\leq\left(  b_{1}\wedge c_{1}\right)  ^{\ast}\leq
b_{2}^{\ast}\vee c_{2}^{\ast}.
\]

\end{lemma}

\begin{proof}
The fact that $\left(  b_{2}^{\ast}\vee c_{2}^{\ast}\right)  \vee\left(
b_{1}\wedge c_{1}\right)  =\top$ implies the right inequality.
\end{proof}

\begin{lemma}
\label{Lem:22}Suppose $a_{i}\in\overline{A}$ and $0\leq s\leq r$ in
$\mathbb{Q}$.

\begin{enumerate}
\item $a_{1}\blacktriangleright r\vee a_{2}\blacktriangleright r=\left(
a_{1}\vee a_{2}\right)  \blacktriangleright r$, and $a_{1}\blacktriangleright
r\wedge a_{2}\blacktriangleright r=\left(  a_{1}\wedge a_{2}\right)
\blacktriangleright r$

\item $a_{1}\blacktriangleleft r\vee a_{2}\blacktriangleleft r=\left(
a_{1}\wedge a_{2}\right)  \blacktriangleleft r$, and $a_{1}\blacktriangleleft
r\wedge a_{2}\blacktriangleleft r=\left(  a_{1}\vee a_{2}\right)
\blacktriangleleft r$.

\item If $s<r$ then $a\blacktriangleright r\prec a\blacktriangleright s$,
hence $a\blacktriangleright s^{\ast}\prec a\blacktriangleright r^{\ast}$.

\item $a\blacktriangleleft s\wedge a\blacktriangleright r=\bot$, and if $s<r$
then $a\blacktriangleleft r\vee a\blacktriangleright s=\top$.

\item If $s<r$ then $a\blacktriangleleft s\prec a\blacktriangleleft r$, hence
$a\blacktriangleleft r^{\ast}\prec a\blacktriangleleft s^{\ast}$.
\end{enumerate}
\end{lemma}

\begin{proof}
(1) By Proposition \ref{Prop:3} we can express $a_{1}\blacktriangleright r\vee
a_{2}\blacktriangleright r$ as
\[
\left[  a_{1}\ominus r\right]  \vee\left[  a_{2}\ominus r\right]  =\left[
a_{1}\ominus r\vee a_{2}\ominus r\right]  =\left[  \left(  a_{1}\vee
a_{2}\right)  \ominus r\right]  =\left(  a_{1}\vee a_{2}\right)
\blacktriangleright r.
\]
The penultimate equality is provided by \cite[3.3.1(8)]{Ball:2013}. The
argument for the second clause of part (1) is similar.

(2)We can express $a_{1}\blacktriangleleft r\wedge a_{2}\blacktriangleleft r$
as
\begin{gather*}
\bigvee_{0<s_{1}<r}\left[  a_{1}\ominus s_{1}\right]  ^{\ast}\wedge
\bigvee_{0<s_{2}<r}\left[  a_{2}\ominus s_{2}\right]  ^{\ast}=\bigvee
_{0<s_{i}<r}\left(  \left[  a_{1}\ominus s_{1}\right]  ^{\ast}\wedge\left[
a_{2}\ominus s_{2}\right]  ^{\ast}\right)  =\\
\bigvee_{0<s_{i}<r}\left(  \left[  a_{1}\ominus s_{1}\right]  \vee\left[
a_{2}\ominus s_{2}\right]  \right)  ^{\ast}=\bigvee_{0<s<r}\left(  \left[
a_{1}\ominus s\right]  \vee\left[  a_{2}\ominus s\right]  \right)  ^{\ast}=\\
\bigvee_{0<s<r}\left[  a_{1}\ominus s\vee a_{2}\ominus s\right]  ^{\ast
}=\bigvee_{0<s<r}\left[  \left(  a_{1}\vee a_{2}\right)  \ominus s\right]
^{\ast}=\left(  a_{1}\vee a_{2}\right)  \blacktriangleleft r.
\end{gather*}
And we can express $a_{1}\blacktriangleleft r\vee a_{2}\blacktriangleleft r$
as
\begin{gather*}
\bigvee_{\mathbb{N}}a_{1}\ominus\left(  r-\frac{1}{n}\right)  ^{\bot}%
\vee\bigvee_{\mathbb{N}}a_{2}\ominus\left(  r-\frac{1}{m}\right)  ^{\bot} =\\
\bigvee_{\mathbb{N}}\left(  a_{1}\ominus\left(  r-\frac{1}{n}\right)  ^{\bot
}\vee a_{2}\ominus\left(  r-\frac{1}{n}\right)  ^{\bot}\right)  =\\
\bigvee_{\mathbb{N}}\left(  a_{1}\ominus\left(  r-\frac{1}{n}\right)  \wedge
a_{2}\ominus\left(  r-\frac{1}{n}\right)  \right)  ^{\bot} =\\
\bigvee_{\mathbb{N}}\left(  \left(  a_{1}\wedge a_{2}\right)  \ominus\left(
r-\frac{1}{n}\right)  \right)  ^{\bot}=\left(  a_{1}\wedge a_{2}\right)
\blacktriangleleft r
\end{gather*}

The second equality is justified by Lemma \ref{Lem:22}, the third by
\cite[3.3.1(7)]{Ball:2013}.

(3) is Lemma \ref{Lem:9}. To check (4), compute
\[
a\blacktriangleright r\wedge a\blacktriangleleft s=a\blacktriangleright
r\wedge\bigvee_{t<s}a\blacktriangleright t^{\ast}=\bigvee_{t<s}\left(
a\blacktriangleright r\wedge a\blacktriangleright t^{\ast}\right)  =\bot.
\]
Then observe that, for $s<r$,
\[
a\blacktriangleright s\vee a\blacktriangleleft r=a\blacktriangleright
s\vee\bigvee_{t<r}a\blacktriangleright t^{\ast}=\bigvee_{s<t<r}\left(
a\blacktriangleright s\vee a\blacktriangleright t^{\ast}\right)  =\top.
\]
The final equality is implied by (3) above. And finally, (5) holds because
$a\blacktriangleright s$ serves as a witness to $a\blacktriangleleft s\prec
a\blacktriangleleft r$, a fact we have established in (3) and (4).
\end{proof}

\begin{lemma}
\label{Lem:19}For $a,b\in\overline{A}$, $a+b\in\overline{A}$ implies $b\in
a\blacktriangleleft1$.
\end{lemma}

\begin{proof}
For any $n$,
\[
a\ominus\left(  1-\frac{1}{n}\right)  \wedge b\ominus\frac{1}{n}\leq\left(
a+b\right)  \ominus1=0,
\]
by \cite[3.3.7]{Ball:2013}, and this implies that
\[
b\ominus\frac{1}{n}\in a\ominus\left(  1-\frac{1}{n}\right)  ^{\bot}%
\subseteq\left[  \bigcup_{s<1}a\ominus s^{\bot}\right]  =a\blacktriangleleft
1.
\]
We conclude that $b\in a\blacktriangleleft1$ by Lemma \ref{Lem:1}.
\end{proof}

\begin{lemma}
\label{Lem:23}For $a,b\in\overline{A}$, $b\blacktriangleright0\vee
a\blacktriangleleft1=\left(  a-b\right)  ^{+}\blacktriangleleft1$.
\end{lemma}

\begin{proof}
By replacing $b$ by $a\wedge b$, we may assume without loss of generality that
$a\geq b$. The kernel $\left(  a-b\right)  ^{+}\blacktriangleleft1$ contains
$a\blacktriangleleft1$ since the map $a\longmapsto a\blacktriangleleft1$ is
order reversing by Lemma \ref{Lem:22}, and $\left(  a-b\right)  ^{+}%
\blacktriangleleft1$ contains $b$ by Lemma \ref{Lem:19}. Consider now a
truncation kernel $K$ which contains both $a\blacktriangleleft1$ and $b$; we
wish to show that
\[
K\supseteq\left(  a-b\right)  ^{+}\blacktriangleleft1=%
{\textstyle\bigvee_{r<1}}
\left(  a-b\right)  ^{+}\ominus r^{\bot},
\]
i.e., $K\supseteq$ $\left(  a-b\right)  ^{+}\ominus r^{\bot}=\left(  a-a\wedge
b\right)  \ominus r^{\bot}$ for all $r<1$. But from \cite[3.3.1(12)]%
{Ball:2013} we know that
\[
\left(  a-a\wedge b\right)  \ominus r+a\wedge b=a\ominus r\vee\left(  a\wedge
b\right)  ,
\]
from which follows $\left(  a-a\wedge b\right)  \ominus r=\left(  a\ominus
r-a\wedge b\right)  ^{+}$, so that, in the end, what we need to show is that
$\left(  a\ominus r-a\wedge b\right)  ^{+\bot}\subseteq K$ for all $r<1$. For
that purpose fix $r$ and consider an element $x\geq0$ such that $x\wedge
\left(  a\ominus r-a\wedge b\right)  ^{+}=0$. Note that, since $\left(
a\ominus r-a\wedge b\right)  ^{+}$ may be considered to be the result of
disjointifying $a\ominus r$ and $a\wedge b$, it follows that
\[
2\left(  \left(  a\ominus r-a\wedge b\right)  ^{+}\vee\left(  a\wedge
b\right)  \right)  \geq a\ominus r\vee\left(  a\wedge b\right)  .
\]
(See the discussion of disjointification preceding \cite[3.3.8]{Ball:2013}.)

We claim that $x\wedge a\ominus r\leq2\left(  a\wedge b\right)  $, for%
\begin{gather*}
x\wedge a\ominus r \leq x\wedge2\left(  \left(  a\ominus r-a\wedge b\right)
^{+}\vee\left(  a\wedge b\right)  \right)  =\\
\left(  x\wedge2\left( a\ominus r-a\wedge b\right)  ^{+}\right)  \vee\left(
x\wedge2\left(  a\wedge b\right)  \right)  =x\wedge2\left(  a\wedge b\right)
\leq2\left(  a\wedge b\right)  .
\end{gather*}
But then we have
\begin{gather*}
\left[  x\right]  =\left[  x\right]  \wedge\top=\left[  x\right]
\wedge\left(  a\blacktriangleright r\vee a\blacktriangleleft1\right)  =\\
\newline\left[  x\wedge a\ominus r\right]  \vee\left(  \left[  x\right]
\wedge a\blacktriangleleft1\right)  \leq\left[  2\left(  a\wedge b\right)
\right]  \vee a\blacktriangleleft1\leq\left[  b\right]  \vee
a\blacktriangleleft1\leq K.
\end{gather*}
The conclusion is that $x\in K$, and this finishes the proof.
\end{proof}

For a convex $\ell$-subgroup $K$ of $A$, we let
\[
^{0}K\equiv\left\{  a\in\overline{A}:a\blacktriangleright0\subseteq K\right\}
\;\text{and\ }^{1}K\equiv\left\{  a\in\overline{A}:a\blacktriangleleft
1\subseteq K\right\}  .
\]
$^{0}K$ is a convex monoid with respect to bounded addition, in the sense that
$\overline{a+b}\in{}^{0}K$ whenever $a,b\in{}^{0}K$. And $^{1}K$ is a filter
on $\overline{A}$ which is disjoint from $^{0}K$ (if $K$ is proper) by Lemma
\ref{Lem:23}. In fact, if $K\in\mathcal{K}A$ then $^{0}K$ is just
$\overline{K}=\left\{  \overline{a}:a\in K^{+}\right\}  $.

\begin{corollary}
\label{Cor:3}Suppose $K\in\mathcal{K}A$. Then $a\in{}^{1}K$ and $b\in{}^{0}K$
imply $\left(  a-b\right)  ^{+}\in{}^{1}K\;$
\end{corollary}

\begin{proof}
This follows directly from Lemma \ref{Lem:23}.
\end{proof}

\begin{corollary}
Suppose $K\in\mathcal{K}A$. Then $K$ contains $a\blacktriangleleft1$ for some
$a\in\overline{A}$ iff $K=\bigcup_{^{1}K}a\blacktriangleleft1$.
\end{corollary}

\begin{proof}
Suppose $a\blacktriangleleft1\subseteq K$ and $b\in K$ for some $a,b\in
\overline{A}$. Then $\left(  a-b\right)  ^{+}\in{}^{1}K$ by Corollary
\ref{Cor:3}, and $b\in\left(  a-b\right)  ^{+}\blacktriangleleft1$ by Lemma
\ref{Lem:19}.
\end{proof}

\section{Representing $A$ in $\mathcal{RK}A$\label{Sec:5}}

We exhibit a natural $\mathbf{AT}$-injection $A\rightarrow\mathcal{RK}A$.
Though intuitive and simple enough, this representation is not functorial. It
is, however, the most important step in the development of a fully functorial
representation of truncs and their morphisms, culminating in Theorems
\ref{Thm:4} and \ref{Thm:3}.

\subsection{ The frame map $\protect\underline{a}$}

\begin{definition}
\label{Def:4}For $a\in A^{+}$, define the map $\underline{a}:\left\{  \left(
r,\infty\right)  :r\in\mathbb{Q}\right\}  \rightarrow\mathcal{K}A$ by the rule%
\[
\underline{a}\left(  r,\infty\right)  \equiv\left\{
\begin{array}
[c]{ll}%
a\blacktriangleright r, & r\geq0\\
\top & r<0
\end{array}
\right.  ,\;r\in\mathbb{Q}.
\]

\end{definition}

\begin{proposition}
Each $\underline{a}$ extends to a unique frame map $\mathcal{O}\mathbb{R}%
\rightarrow\mathcal{K}A$, which we also denote $\underline{a}$.
\end{proposition}

\begin{proof}
According to \cite[3.1.2]{BallHager:1991}, this amounts to establishing three things.

\begin{itemize}
\item $\underline{a}\left(  s,\infty\right)  \prec\underline{a}\left(
r,\infty\right)  $ for $r<s$ in $\mathbb{Q}$.

\item $\underline{a}\left(  r,\infty\right)  =\bigvee_{r<s}\underline{a}%
\left(  s,\infty\right)  $ for all $r\in\mathbb{Q}$.

\item and $\bigvee_{\mathbb{Q}}\underline{a}\left(  r,\infty\right)
=\bigvee_{\mathbb{Q}}\underline{a}\left(  r,\infty\right)  ^{\ast}=\top$.
\end{itemize}

\noindent The first point is the content of Lemma \ref{Lem:22}(3), the second
is Corollary \ref{Cor:4}, and the third point follows from the claim that
$\bigvee_{\mathbb{N}}\left[  a\ominus n\right]  ^{\ast}=\top$ in
$\mathcal{K}A$. To establish this claim, in turn, it suffices to show that
\[
\forall~b\in A^{+}~\exists~c\in A^{+}~\forall~n\in\mathbb{N}~\exists
~i\in\mathbb{N}~\left(  \left(  n\overline{b}-c\right)  ^{+}\wedge a\ominus
i=0\right)  .
\]
For, when we denote the convex $\ell$-subgroup $\bigcup_{\mathbb{N}}\left[
a\ominus n\right]  ^{\ast}$ by $B$, the condition displayed above simply
asserts that $\overline{b}$ lies in the archimedean kernel generated by $B$ by
virtue of satisfying Lemma \ref{Lem:4}. And if so, of course, it follows that
$b\in\left[  B\right]  $ by Lemma \ref{Lem:5}. But satisfying the displayed
condition is easy. Given $b\in A^{+}$, take $c$ to be $a$, and, upon being
presented with $n$, take $i$ to be $n$. The condition becomes
\[
\left(  n\overline{b}-a\right)  ^{+}\wedge a\ominus n=n\left(  \left(
\overline{b}-a/n\right)  ^{+}\wedge\left(  a/n-\overline{a/n}\right)  \right)
=0.
\]
This follows from \cite[3.3.1(1)]{Ball:2013}, with $a$ and $b$ there taken to
be $a/n$ and $b$ here.
\end{proof}

\begin{lemma}
\label{Lem:15}For $a\in A^{+}$ and $0\leq r<1$ in $\mathbb{Q}$, $\left[
\overline{a}\ominus r\right]  =\left[  a\ominus r\right]  =\left[
\overline{a\ominus r}\right]  .$
\end{lemma}

\begin{proof}
In \cite[3.3.5]{Ball:2013}, take $p$, $q$, and $a$ there to be $r$, $1-r$, and
$a/r$ here, to get
\begin{gather*}
\overline{a}=r\overline{a/r}+\left(  1-r\right)  \overline{r\left(
a/r-\overline{a/r}\right)  /\left(  1-r\right)  }=\\
r\overline{\overline{a}/r}+\left(  1-r\right)  \overline{r\left(
a/r\ominus1\right)  /\left( 1-r\right)  }= r\overline{\overline{a}/r}+\left(
1-r\right)  \overline{\left(  a\ominus r\right)  /\left(  1-r\right)  }.
\end{gather*}
This rearranges to $\overline{a}\ominus r=\overline{a}-r\overline{\overline
{a}/r}=\left(  1-r\right)  \overline{\left(  a\ominus r\right)  /\left(
1-r\right)  }$, hence
\[
\left[  \overline{a}\ominus r\right]  =\left[  \left(  1-r\right)
\overline{\left(  a\ominus r\right)  /\left(  1-r\right)  }\right]  =\left[
\overline{\left(  a\ominus r\right)  /\left(  1-r\right)  }\right]  =\left[
\left(  a\ominus r\right)  /\left(  1-r\right)  \right]  .
\]
Since any nonzero multiple of a generator of a truncation ideal is itself a
generator, this works out to $\left[  a\ominus r\right]  =\left[
\overline{a\ominus r}\right]  $.
\end{proof}

\begin{lemma}
\label{Lem:18}For $a\in A^{+}$ and $r\in\mathbb{F}$,%
\[
\underline{a}\left(  -\infty,r\right)  =\underline{-a}\left(  -r,\infty
\right)  =\left\{
\begin{array}
[c]{ll}%
a\blacktriangleleft r, & r>0\\
\bot & r\leq0
\end{array}
\right.
\]

\end{lemma}

\begin{proof}
In light of the fact that $a\blacktriangleleft r=\bigvee_{0<s<r}\left[
a\ominus s\right]  ^{\ast}$ for $r\geq0$, this is an application of a general
principle (\cite[3.1.1, 3.1.3]{BallHager:1991}): for a frame $L$, a frame map
$f\in\mathcal{R}L$, and for $r\in\mathbb{Q}$,
\[
f\left(  -\infty,r\right)  =\left(  -f\right)  \left(  -r,\infty\right)
=\bigvee_{s<r}f\left(  s,\infty\right)  ^{\ast}.
\]
In the present situation, take $f\equiv\underline{a}$.
\end{proof}

We denote the underscore map $a\longmapsto\underline{a}$ by $\kappa
_{A}:A\rightarrow\mathcal{RK}A$, and we denote its range by
\[
\underline{A}\equiv\left\{  \underline{a}:a\in A\right\}  \subseteq
\mathcal{RK}A.
\]
We show in Theorem \ref{Thm:6} that $\kappa_{A}$ is an isomorphism
$A\rightarrow\underline{A}$. A little ground clearing is necessary first.

\subsection{$\kappa_{A}$ is an isomorphism}

We begin by showing that $\kappa$ preserves truncation and diminution.

\begin{lemma}
\label{Lem:11}For any $a\in A^{+}$,
\[
\underline{a\ominus1}=\left(  \underline{a}-1\right)  ^{+}\;\text{and
\ }\underline{\overline{a}}=\underline{a}\wedge1.
\]
It follows that $\underline{a\ominus r}=\left(  \underline{a}-r\right)  ^{+}$
and $r\overline{a/r}=a\wedge r$ for $r\in\mathbb{Q}$.
\end{lemma}

Reader beware, for Lemma \ref{Lem:11} can easily be misunderstood. The $1$ and
the $r$ which appear on the left sides of the first and third equations,
respectively, refer to scalars, while the same symbols on the right sides of
the first three equations refer to the corresponding constant frame functions
$\mathcal{O}\mathbb{R}\rightarrow\mathcal{K}A$. Such a constant function is
given by the rule%
\[
r\left(  s,\infty\right)  =\left\{
\begin{array}
[c]{cc}%
\bot, & s\geq r\\
\top, & s<r
\end{array}
\right.  ,\text{\ or }r\left(  -\infty,s\right)  =\left\{
\begin{array}
[c]{cc}%
\top, & s>r\\
\bot, & s\leq r
\end{array}
\right.  .
\]
\emph{This constant function need not lie in }$\emph{\underline{\emph{A}}}%
$\emph{.}

\begin{proof}
We have, for $r\in\mathbb{Q}$,
\[
\underline{a\ominus1}\left(  r,\infty\right)  =\left\{
\begin{array}
[c]{ll}%
\left[  \left(  a\ominus1\right)  \ominus r\right]  , & r\geq0\\
\top, & r<0
\end{array}
\right.  =\left\{
\begin{array}
[c]{ll}%
\left[  a\ominus\left(  1+r\right)  \right]  , & r\geq0\\
\top & r<0
\end{array}
\right.
\]
On the other hand we have
\[
\left(  \underline{a}-1\right)  ^{+}\left(  r,\infty\right)  =\left(
\underline{a}-1\right)  \left(  r,\infty\right)  \vee0\left(  r,\infty\right)
=\underline{a}\left(  r+1,\infty\right)  \vee0\left(  r,\infty\right)  ,
\]
with the equalities justified by 3.1.3 and 3.1.5, respectively, of
\cite[3.1.3]{BallHager:1991}. Since%
\[
\underline{a}\left(  r+1,\infty\right)  =\left\{
\begin{array}
[c]{ll}%
\left[  a\ominus\left(  r+1\right)  \right]  , & r+1\geq0\\
\top & r+1<0
\end{array}
\right.  \;\text{and\ }0\left(  r,\infty\right)  =\left\{
\begin{array}
[c]{ll}%
\bot, & r\geq0\\
\top & r<0
\end{array}
\right.  ,
\]
we see by inspection that the $\underline{a\ominus1}\left(  r,\infty\right)
=\left(  \underline{a}-1\right)  ^{+}\left(  r,\infty\right)  $ for all
$r\in\mathbb{Q}$. From this follows%
\[
\underline{\overline{a}}=\underline{a}-\underline{a\ominus1}=\underline{a}%
-\left(  \underline{a}-1\right)  ^{+}=\underline{a}+\left(  1-\underline{a}%
\right)  \wedge0=\underline{a}\wedge 1. \qedhere
\]

\end{proof}

\begin{theorem}
\label{Thm:6}$\underline{A}$ is a subtrunc of $\mathcal{RK}A$, and $\kappa
_{A}\equiv\left(  a\longmapsto\underline{a}\right)  $ is a trunc isomorphism
$A\rightarrow\underline{A}$.
\end{theorem}

\begin{proof}
It is folklore that, for $\ell$-groups $B$ and $C$, any map $B^{+}\rightarrow
C^{+}$ which preserves meets and sums extends to a unique $\ell$-homomorphism
$B\rightarrow C$. Therefore we need only check that the restriction of the
underscore map to $A^{+}$ preserves meets and sums. So consider $a_{i}\in
A^{+}$ and $r\in\mathbb{F}$.
\begin{gather*}
\left(  \underline{a_{1}}\wedge\underline{a_{2}}\right)  \left(
r,\infty\right)  = \underline{a_{1}}\left(  r,\infty\right)  \wedge
\underline{a_{2}}\left(  r,\infty\right)  = \left\{
\begin{array}
[c]{ll}%
\left[  a_{1}\ominus r\right]  \wedge\left[  a_{2}\ominus r\right]  , &
r\geq0\\
\top, & r<0
\end{array}
\right.  =\\
\left\{
\begin{array}
[c]{ll}%
\left[  a_{1}\ominus r\wedge a_{2}\ominus r\right]  , & r\geq0\\
\top, & r<0
\end{array}
\right.  = \left\{
\begin{array}
[c]{ll}%
\left[  \left(  a_{1}\wedge a_{2}\right)  \ominus r\right]  , & r\geq0\\
\top, & r<0
\end{array}
\right.  = \underline{a_{1}\wedge a_{2}}\left(  r,\infty\right)  .
\end{gather*}

The first equality is justified by \cite[3.1.3]{BallHager:1991}, and the
fourth by \cite[3.3.1(7)]{Ball:2013}.

Again fix $a_{i}\in A^{+}$, and consider
\[
\left(  \underline{a_{1}}+\underline{a_{2}}\right)  \left(  r,\infty\right)
=\bigvee_{U_{1}+U_{2}\subseteq\left(  r,\infty\right)  }\left(
\underline{a_{1}}\left(  U_{1}\right)  \wedge\underline{a_{2}}\left(
U_{2}\right)  \right)  =\bigvee_{s}\left(  \underline{a_{1}}\left(
s,\infty\right)  \wedge\underline{a_{2}}\left(  r-s,\infty\right)  \right)
\]
The second equality is justified by the observation that if open subsets
$U_{i}\subseteq\mathbb{R}$ satisfy $U_{1}+U_{2}\subseteq\left(  r,\infty
\right)  $ then $U_{1}$ must be bounded below, say by $s$, in which case
$U_{2}$ must be bounded below by $r-s$. If $s<0$ then the corresponding term
of the join satisfies%
\[
\underline{a_{1}}\left(  s,\infty\right)  \wedge\underline{a_{2}}\left(
r-s,\infty\right)  =\underline{a_{2}}\left(  r-s,\infty\right)  \leq
\underline{a_{2}}\left(  r,\infty\right)  \leq\underline{a_{1}+a_{2}}\left(
r,\infty\right)  ,
\]
and if $r-s<0$ then the term satisfies
\[
\underline{a_{1}}\left(  s,\infty\right)  \wedge\underline{a_{2}}\left(
r-s,\infty\right)  =\underline{a_{1}}\left(  s,\infty\right)  \leq
\underline{a_{1}}\left(  r,\infty\right)  \leq\underline{a_{1}+a_{2}}\left(
r,\infty\right)  .
\]
In the only remaining case we have $0\leq s\leq r$, which gives%
\begin{gather*}
\underline{a_{1}}\left(  s,\infty\right)  \wedge\underline{a_{2}}\left(
r-s,\infty\right)  =\left[  a_{1}\ominus s\right]  \wedge\left[  a_{2}%
\ominus\left(  r-s\right)  \right]  =\\
\left[  a_{1}\ominus s\wedge a_{2}\ominus\left(  r-s\right)  \right]
\leq\left[  \left(  a_{1}+a_{2}\right)  \ominus r\right]  =\underline{a_{1}%
+a_{2}}\left(  r,\infty\right)  .
\end{gather*}
The inequality holds by \cite[3.3.7]{Ball:2013}. Thus have we established that
$\left(  \underline{a_{1}}+\underline{a_{2}}\right)  \left(  r,\infty\right)
\leq\underline{a_{1}+a_{2}}\left(  r,\infty\right)  $.

To establish the opposite inequality it is enough to show that, for any
$\varepsilon>0$ and $r\geq0$,%
\[
\left(  \underline{a_{1}}+\underline{a_{2}}\right)  \left(  r,\infty\right)
\geq\underline{a_{1}+a_{2}}\left(  r+2\varepsilon,\infty\right)  .
\]
To that end fix $\varepsilon$ and $r$, and put $x\equiv\underline{a_{1}+a_{2}%
}\left(  r+2\varepsilon,\infty\right)  $. Then we have
\begin{gather*}
x   =
x\wedge\top = 
x\wedge\bigvee_{r_{1}}\underline{a_{1}}\left(r_{1}-\varepsilon,r_{1}+\varepsilon\right)  = 
\bigvee_{r_{1}}\left(x\wedge\underline{a_{1}}\left(  r_{1}-\varepsilon,r_{1}+\varepsilon\right)\right) = \\
\bigvee_{r_{1}}\left(  \left(  x\wedge\underline{a_{1}}\left(
-\infty,r_{1}+\varepsilon\right)  \right)  \wedge\underline{a_{1}}\left(
r_{1}-\varepsilon,\infty\right)  \right)   \leq \\ 
\bigvee_{r_{1}}\left(  \underline{a_{2}}\left(  r+\varepsilon
-r_{1},\infty\right)  \wedge\underline{a_{1}}\left(  r_{1}-\varepsilon
,\infty\right)  \right) = 
\left(  \underline{a_{1}}+\underline{a_{2}}\right)  \left(  r,\infty\right)  .
\end{gather*}
The inequality is justified by the observation that
\begin{gather*}
x\wedge\underline{a_{1}}\left(  -\infty,r_{1}+\varepsilon\right)
=x\wedge\underline{-a_{1}}\left(  -r_{1}-\varepsilon,\infty\right)  =\\
\underline{a_{1}+a_{2}}\left(  r+2\varepsilon,\infty\right)  \wedge
\underline{a_{1}}\left(  -\infty,r_{1}+\varepsilon\right)  =\\
\bigvee_{r_{2}}\left(  \underline{a_{1}}\left(  r_{2}+\varepsilon
,\infty\right)  \wedge\underline{a_{2}}\left(  r+\varepsilon-r_{2}%
,\infty\right)  \right)  \wedge\underline{a_{1}}\left(  -\infty,r_{1}%
+\varepsilon\right)  =\\
\bigvee_{r_{2}}\left(  \underline{a_{1}}\left(  r_{2}+\varepsilon
,\infty\right)  \wedge\underline{a_{2}}\left(  r+\varepsilon-r_{2}%
,\infty\right)  \wedge\underline{a_{1}}\left(  -\infty,r_{1}+\varepsilon
\right)  \right)  =\\
\bigvee_{r_{2}}\left(  \underline{a_{1}}\left(  r_{2}+\varepsilon
,r_{1}+\varepsilon\right)  \wedge\underline{a_{2}}\left(  r+\varepsilon
-r_{2},\infty\right)  \right)  .
\end{gather*}
For if $r_{2}\geq r_{1}$ then the contribution of the corresponding term to
the last join is trivial, and if $r_{2}<r_{1}$ then the corresponding term
satisfies
\[
\underline{a_{1}}\left(  r_{2}+\varepsilon,r_{1}+\varepsilon\right)
\wedge\underline{a_{2}}\left(  r+\varepsilon-r_{2},\infty\right)
\leq\underline{a_{2}}\left(  r+\varepsilon-r_{2},\infty\right)  \leq
\underline{a_{2}}\left(  r+\varepsilon-r_{1},\infty\right)  ,
\]
with the result that
\begin{align*}
x\wedge\underline{a_{1}}\left(  -\infty,r_{1}+\varepsilon\right)   &
=\bigvee_{r_{2}}\left(  \underline{a_{1}}\left(  r_{2}+\varepsilon
,r_{1}+\varepsilon\right)  \wedge\underline{a_{2}}\left(  r+\varepsilon
-r_{2},\infty\right)  \right) \\
&  \leq\underline{a_{2}}\left(  r+\varepsilon-r_{1},\infty\right)  .
\end{align*}
It is obvious that $a\longmapsto\underline{a}$ is one-one, since
\[
0<a\in A\Longrightarrow\underline{a}\left(  0,\infty\right)  \equiv\left[
a\right]  \neq\bot=\underline{0}\left(  0,\infty\right)  .
\]
Finally, the fact that $\overset{\_}{\rightarrow}$ preserves truncation is the
content of Lemma \ref{Lem:11}.
\end{proof}

\subsection{The representation $A\rightarrow\mathcal{RK}A$ is not functorial}

To say that the representation $\kappa_{A}$ is functorial is to say that
$\left(  \kappa_{A},\mathcal{K}A\right)  $ constitutes an $\mathcal{R}%
$-universal arrow with domain $A$. \begin{figure}[tbh]
\setlength{\unitlength}{4pt}
\par
\begin{center}
\begin{picture}(48,13)(0,-1)
\small
\put(0,12){\makebox(0,0){$A$}}
\put(0,0){\makebox(0,0){$B$}}
\put(12,0){\makebox(0,0){$\mathcal{RK} B$}}
\put(12,12){\makebox(0,0){$\mathcal{RK}A$}}
\put(24,12){\makebox(0,0){$\mathcal{K}A$}}
\put(24,0){\makebox(0,0){$\mathcal{K}B$}}
\put(36,6){\makebox(0,0){$\mathcal{O}\mathbb{R}$}}
\put(2,12){\vector(1,0){6}}
\put(2,0){\vector(1,0){6}}
\put(0,10){\vector(0,-1){8}}
\put(12,10){\vector(0,-1){8}}
\put(24,10){\vector(0,-1){8}}
\put(34,8){\vector(-2,1){6}}
\put(34,4){\vector(-2,-1){7}}
\put(15,6){\makebox(0,0){$\mathcal{R}g$}}
\put(22.5,6){\makebox(0,0){$g$}}
\put(4.5,13.5){\makebox(0,0){$\kappa_A$}}
\put(5,-1.5){\makebox(0,0){$\kappa_B$}}
\put(-2,6){\makebox(0,0){$\theta$}}
\put(32.5,10.5){\makebox(0,0){$\widehat{a}$}}
\put(33,1.5){\makebox(0,0){$\theta (a)$}}
\end{picture}
\end{center}
\end{figure}

\noindent That means that for any other trunc morphism $\theta$, with a
codomain of the form $\mathcal{R}L$ for some frame $L$, there is a frame
morphism $g$ such that $\mathcal{R}g\circ\mu=\theta$. But the arrow $\left(
\kappa_{A},\mathcal{K}A\right)  $ is no such thing, as we can see from the
following simple example

\begin{example}
\label{Ex:1}Let $A\equiv\mathbb{R}$, so that $\mathcal{K}A$ is the two-element
frame $2\equiv\left\{  \bot,\top\right\}  $. A little reflection on the
definitions leads to the conclusion that, for each $a\in\mathbb{R}$,
$\underline{a}=\kappa_{A}\left(  a\right)  $ is the constant $a$ function,
i.e.,
\[
\underline{a}\left(  U\right)  =\left\{
\begin{array}
[c]{ccc}%
\top & , & a\in U\\
\bot & , & a\notin U
\end{array}
\right.  ,\;U\in\mathcal{O}\mathbb{R}.
\]

Now let $B\equiv\mathbb{R}^{2}=\mathbb{R}\times\mathbb{R}$ with cardinal
order, so that $\mathcal{K}B$ is the four element frame $4\equiv\left\{
\left(  \bot,\bot\right)  <\left(  \bot,\top\right)  ,\left(  \top
,\bot\right)  <\left(  \top,\top\right)  \right\}  $. (Here $\left(  \bot
,\top\right)  $ represents $0\times\mathbb{R}$, $\left(  \top,\bot\right)  $
represents $\mathbb{R}\times0$, etc. ) A little more reflection reveals that
\[
\underline{\left(  a,b\right)  }\left(  U\right)  =\left\{
\begin{array}
[c]{lll}%
\left(  \top,\top\right)  & , & a,b\in U\\
\left(  \bot,\top\right)  & , & a\notin U\ni b\\
\left(  \top,\bot\right)  & , & b\notin U\ni a\\
\left(  \bot,\bot\right)  & , & a,b\notin U
\end{array}
\right.  ,\;U\in\mathcal{O}\mathbb{R}.
\]

Finally, consider the embedding $\theta:A\rightarrow B\equiv\left(
a\longmapsto\left(  a,0\right)  \right)  $. Functoriality would require the
existence of a frame map $g:\mathcal{K}A\rightarrow\mathcal{K}B$ such that
$\mathcal{R}g\circ\kappa_{A}=\kappa_{B}\circ\theta$, i.e., such that%
\[
g\circ\underline{a}\left(  U\right)  =\underline{\theta\left(  a\right)
}\left(  U\right)  ,\;U\in\mathcal{O}\mathbb{R}.
\]
However, there is exactly one frame map $2\rightarrow4$, and it is easy to
check that it does not have the property displayed.

Note that $\theta$ is an example of a trunc morphism between 
$\mathbf{W}$-objects which is not a $\mathbf{W}$-morphism. 
Thus $\mathbf{W}$ is not a full subcategory of $\mathbf{AT}$. 
\end{example}

It is perhaps surprising that functoriality can be restored to the
representation by the simple expedient of adjoining a point to each frame
under consideration, and requiring the frame maps to commute with the
designated points. Nevertheless this is the case, but to prove it requires a
little spadework.

\section{Pointed and filtered frames\label{Sec:1}}

We digress to fill in the background necessary to systematically adjoin a
single point to each frame under consideration. This procedure is not only
necessary for our development, but it is also of interest in its own right. We
remind the reader of our convention that all frames are completely regular
unless explicitly stipulated otherwise. We denote the two-element frame by
$2\equiv\left\{  \bot,\top\right\}  $.

\subsection{Pointed frames}

A \emph{pointed frame} is a pair $\left(  L,\ast_{L}\right)  $, where $L$ is a
frame and $\ast:L\rightarrow2$ is a point of $L$. A \emph{pointed frame
morphism} $f:\left(  L,\ast_{L}\right)  \rightarrow\left(  M,\ast_{M}\right)
$ is a frame morphism $f:L\rightarrow M$ which commutes with the points, i.e.,
such that $\ast_{M}\circ f=\ast_{L}$. We use $\mathbf{pF}$ to denote the
category of pointed frames and their morphisms. Of particular importance is
the\emph{\ frame of the pointed reals} $\mathcal{O}_{0}\mathbb{R}\equiv\left(
\mathcal{O}\mathbb{R},0\right)  $, where $0:\mathcal{O}\mathbb{R}\rightarrow2$
is the constant $0$ frame map.

Perhaps the most natural example of a pointed frame is one of the form
$2L\equiv\left(  2\times L,\ast\right)  $, where $L$ is a frame, $2\times L$
is the ordinary frame product, and $\ast$ is the first projection map $\left(
\varepsilon,a\right)  \longmapsto\varepsilon$, $\varepsilon=0,1$. In fact, we
assume the point map to be this projection whenever we deal with a subobject
of $2L$.

It is worth mentioning that the second projection
\[
\pi:2L\rightarrow L\equiv\left(  \left(  \varepsilon,a\right)  \longmapsto
a,\;a\in L\right)
\]
is the co-free pointed frame over $L$. More explicitly, let $\mathcal{U}$ be
the forgetful functor which assigns to each pointed frame its underlying plain
frame. Then, given a pointed frame $\left(  M,\ast_{M}\right)  $ and a frame
map $f:\mathcal{U}M\rightarrow L$, there exists a unique pointed frame
morphism $k:\left(  M,\ast_{M}\right)  \rightarrow2L$ such that $\pi
\circ\mathcal{U}k=f$. In fact, $\mathcal{U}k$ is just the product map
$\ast_{M}\times f$.

We should also point out that the designated point $\ast$ of $2L$ is isolated.
A point $p$ of a frame $M$ is said to be \emph{isolated} if it is the open
quotient of a complemented element. An \emph{isolated point frame} is a
pointed frame whose designated point is isolated. Isolated point frames are
central to the representation of $\mathbf{W}$-objects as truncs; see Section
\ref{Sec:3}. We use $\mathbf{ipF}$ to denote the full subcategory of
$\mathbf{pF}$ composed of the isolated pointed frames.

\subsection{The standard representation of a pointed frame}

In any frame $M$, we define the \emph{kernel} of a point $\ast:M\rightarrow2$
to be
\[
p\equiv\bigvee_{\ast\left(  a\right)  =\bot}a,
\]
the largest element of $M$ sent to $\bot$ by $\ast$. If $M$ is regular then
$p$ is maximal and $\ast$ is just the closed quotient determined by $p$. The
associated congruence is complemented in the congruence frame of $M$ by the
open quotient determined by $p$, which we denote $\pi$. One concrete
realization of the two quotients is
\[
\ast:M \to {\uparrow} p = \left(  a\mapsto a\vee p\right)
\text{ and } \pi : M \to {\downarrow} a = \left(  a\mapsto a\wedge p\right)  .
\]

If $\left(  M,\ast_M\right)  $ is a pointed frame then we refer to the kernel of
$\ast_M$ as the \emph{kernel of }$M$, and write it $p_{M}$. We denote the
associated closed and open quotient maps by $\ast_{M}:M \to 2$ and
$\pi_{M}:M\to Q_{M}$, respectively, and refer to them as canonical.
The induced product map $\ast_{M}\times\pi_{M}$, 
which we denote by $\nu_{M}:M\rightarrow2Q_{M}$, is a
$\mathbf{pF}$-morphism by construction, and, since the associated congruences
meet to the identity in the congruence lattice, $v_{M}$ is injective. We drop
the subscripts whenever doing so introduces no ambiguity.

\begin{figure}[htb]
\setlength{\unitlength}{4pt}
\par
\begin{center}
\begin{picture}(12,12)(0,0)
\small
\put(0,12){\makebox(0,0){$2Q$}}
\put(12,12){\makebox(0,0){$2$}}
\put(0,0){\makebox(0,0){$Q$}}
\put(12,0){\makebox(0,0){$M$}}

\put(2,12){\vector(1,0){8}}
\put(0,10){\vector(0,-1){8}}
\put(12,2){\vector(0,1){8}}
\put(10,0){\vector(-1,0){8}}
\put(10,2){\vector(-1,1){8}}

\put(6,13){\makebox(0,0){$*$}}
\put(-2,6){\makebox(0,0){$\pi$}}
\put(14.5,6){\makebox(0,0){$\ast_M$}}
\put(6.5,-1.5){\makebox(0,0){$\pi_M$}}
\put(7.5,7){\makebox(0,0){$\nu_M$}}
\end{picture}
\end{center}
\end{figure}

Let $L$ be a frame and $F$ a filter on $L$. We say that $F$ is \emph{regular}
if $\bigvee_{F}b^{\ast}=\top$. Define
\[
2_{F}L\equiv\left\{  \left(  \varepsilon,a\right)  \in2L:\varepsilon
=\top\Longrightarrow a\in F\right\}  .
\]

\begin{lemma}
For any frame $L$ and filter $F$ on $L$, $2_{F}L$ is a sub-pointed frame of
$2L$ which is regular iff $F$ is regular.
\end{lemma}

\begin{proof}
Suppose $2_{F}L$ is regular, and, in order to verify that $F$ is regular,
consider $a<\top$ in $L$. Then $\left(  \bot,a\right)  <p\equiv\left(
\bot,\top\right)  $ in $2_{F}L$, so that by regularity there must be some
$q\prec p$ such that $q\nleq\left(  \bot,a\right)  $, say $q\wedge r=\bot$ and
$r\vee p=\top$ for some $r\in2_{F}L$. Now $q$ must be of the form $\left(
\bot,b\right)  $ since $q\leq p$, hence $r$ must be of the form $\left(
\top,c\right)  $ since $q\vee r=\top$. But then the definition of $2_{F}L$
forces $c\in F$, and $q\wedge r=\bot$ forces $b\wedge c=\bot$, with the result
that $c^{\ast}\nleq a$. That is, $F$ is regular.

Now suppose that $F$ is regular. We claim that, for any $c\in L$,
\[
\left(  \bot,c\right)  =\bigvee\left\{  \left(  \bot,a\wedge b^{\ast}\right)
:b\in F,~a\prec c\right\}  .
\]
This is true because%
\begin{gather*}
\bigvee_{b\in F,~a\prec c}\left(  \bot,a\wedge b^{\ast}\right)  = \left(
\bot,\bigvee_{b\in F,a\prec c}\left(  a\wedge b^{\ast}\right)  \right)  =
\left(  \bot,\bigvee_{a\prec c}\bigvee_{b\in F}\left(  a\wedge b^{\ast
}\right)  \right)  =\\
\left(  \bot,\bigvee_{a\prec c}\left(  a\wedge\bigvee_{b\in F}b^{\ast}\right)
\right)  =\left(  \bot,\bigvee_{a\prec c}a\right)  =\left(  \bot,c\right)  .
\end{gather*}
And $\left(  \top,a^{\ast}\vee b\right)  $ witnesses $\left(  \bot,a\wedge
b^{\ast}\right)  \prec\left(  \bot,c\right)  $, for
\[
\left(  \top,a^{\ast}\vee b\right)  \wedge\left(  \bot,a\wedge b^{\ast
}\right)  =\left(  \bot,\bot\right)  \;\text{and\ }\left(  \top,a^{\ast}\vee
b\right)  \vee\left(  \bot,c\right)  =\left(  \top,\top\right)  .
\]
On the other hand, it is obvious that $\left(  \top,b\right)  =\left(
\top,\bigvee_{a\prec b}a\right)  =\bigvee_{a\prec b}\left(  \top,a\right)  $
for any $b\in F$, and if $a\prec b$ it is just as clear that $\left(
\top,a\right)  \prec\left(  \top,b\right)  $. This shows that $2_{F}L$ is regular.
\end{proof}

\begin{proposition}
\label{Prop:2}Let $M$ be a pointed frame with canonical quotient maps $\ast$
and $\pi:M\rightarrow Q$, and let
\[
F\equiv\left\{  \pi\left(  a\right)  :\ast\left(  a\right)  =\top\right\}  .
\]
Then $F$ is a regular filter on $Q$, and the range $\nu_{M}\left[  M\right]  $
of $\nu_{M}$ is $2_{F}Q$\begin{figure}[tbh]
\setlength{\unitlength}{4pt}
\par
\begin{center}
\begin{picture}(18,18)(0,.5)
\small
\put(0,18){\makebox(0,0){$2L$}}
\put(9,9){\makebox(0,0){$2_F L$}}
\put(18,18){\makebox(0,0){$2$}}
\put(0,0){\makebox(0,0){$L$}}
\put(18,0){\makebox(0,0){$M$}}
\put(2.5,18){\vector(1,0){13.5}}
\put(0,16){\vector(0,-1){14}}
\put(7,11){\vector(-1,1){5}}
\put(16,2){\vector(-1,1){5}}
\put(18,2){\vector(0,1){14}}
\put(16,0){\vector(-1,0){14}}
\put(9,19.5){\makebox(0,0){$\ast$}}
\put(-2.5,9){\makebox(0,0){$\pi$}}
\put(9,-2){\makebox(0,0){$\rho$}}
\put(21,9){\makebox(0,0){$\ast_M$}}
\end{picture}
\end{center}
\end{figure}
\end{proposition}

We denote the range restriction of $\nu_{M}$ by $\tau_{M}:M\rightarrow2_{F}Q$,
and the insertion $2_{F}Q\rightarrow2Q$ by $\sigma_{M}$, so that $\nu
_{M}=\sigma_{M}\tau_{M}$. We refer to $\tau_{M}$ as the \emph{standard
representation of }$M$.

\subsection{The trunc $\mathcal{R}_{0}M$}

For a pointed frame $\left(  M,\ast\right)  $, we denote the family of pointed
frame morphisms $\mathcal{O}_{0}\mathbb{R}\rightarrow M$ by
\[
\mathcal{R}_{0}M\equiv\hom_{\mathbf{pF}}\left(  \mathcal{O}_{0}\mathbb{R}%
,M\right)  =\left\{  a\in\mathcal{R}M:\ast_{M}\circ a=0\right\}  .
\]

\begin{lemma}
\label{Lem:17}$\mathcal{R}_{0}M$ is a subtrunc of $\mathcal{R}M$. However, it
is a not a $\mathbf{W}$-subobject of $\mathcal{R}M$ if $M$ is nontrivial,
i.e., if $\bot<\top$. In fact, the only constant function of $\mathcal{R}M$
present in $\mathcal{R}_{0}M$ is the $0$ function.
\end{lemma}

\begin{proof}
For $0\leq a_{i}\in\mathcal{R}_{0}M$ we have%
\[
\left(  a_{1}+a_{2}\right)  \left(  r,\infty\right)  =
\hspace{-7pt}\bigvee_{U_{1}+U_{2}\subseteq\left(  r,\infty\right)  }\hspace{-7pt}
\left(  a_{1}\left(  U_{1}\right) \wedge a_{2}\left(  U_{2}\right)  \right) =  \hspace{-4pt}\bigvee_{r_{1}+r_{2}=r}\hspace{-4pt}
\left(a_{1}\left(  r_{1},\infty\right)  \wedge a_{2}\left(  r_{2},\infty\right)
\right)  ,
\]
for if nonempty open subsets $U_{i}\subseteq\mathbb{R}$ satisfy $U_{1}%
+U_{2}\subseteq\left(  r,\infty\right)  $ then each $U_{i}$ must be bounded
below, and if $r_{i}=\bigwedge U_{i}$ then $r_{1}+r_{2}\geq r$. Therefore
\[
\ast\left(  a_{1}+a_{2}\right)  \left(  r,\infty\right)  =
\hspace{-4pt}\bigvee_{r_{1}+r_{2}=r}\hspace{-4pt}
\left(  \ast a_{1}\left(  r_{1},\infty\right)  \wedge\ast a_{2}\left(  r_{2},\infty\right)  \right)  =
\left\{
\begin{array}
[c]{cc}%
\top, & r<0\\
\bot, & r\geq0
\end{array}
\right.  =0.
\]
The verifications that $\mathcal{R}_{\mathbf{pF}}L$ is closed under negation,
meet, and join all go along similar lines. Finally, for $0\leq a\in
\mathcal{R}_{0}M$ and $1\in\mathcal{R}M$ it is clear that $a\wedge
1\in\mathcal{R}_{0}M$, since for all $r\in\mathbb{R}$ we have
\begin{align*}
\ast\circ\left(  a\wedge1\right)  \left(  r,\infty\right)   &  =\ast\left(
a\left(  r,\infty\right)  \wedge1\left(  r,\infty\right)  \right)  =\ast
a\left(  r,\infty\right)  \wedge\ast1\left(  r,\infty\right) \\
&  =0\left(  r,\infty\right)  \wedge1\left(  r,\infty\right)  =
0\left(r,\infty\right) \qedhere
\end{align*}
\end{proof}

Thanks to Lemma \ref{Lem:17}, we can, and do, regard $\mathcal{R}_{0}M$ as a
trunc. In fact, our main result is that objects of this form are universal for
$\mathbf{AT}$; see Theorems \ref{Thm:4} and \ref{Thm:3}. We emphasize that
these are virtually never $\mathbf{W}$-objects.

On the other hand, the universal $\mathbf{W}$-objects are those of the form
$\mathcal{R}L$, $L$ a frame. It is therefore of interest to learn that each
such $\mathcal{R}L$ is $\mathbf{AT}$-isomorphic to $\mathcal{R}_{0}2L$.

\begin{proposition}
Let $L$ be a frame, and let $2L\equiv2\times L$ be the frame product with
projection $\pi:2L\rightarrow L$. For any $a\in\mathcal{R}L$, let
$\widehat{a}$ be the induced product map $0\times a$. \begin{figure}[tbh]
\setlength{\unitlength}{4pt}
\par
\begin{center}
\begin{picture}(12,12)
\small
\put(0,12){\makebox(0,0){$2L$}}
\put(12,0){\makebox(0,0){$\mathcal{O}\mathbb{R}$}}
\put(12,12){\makebox(0,0){$2$}}
\put(0,0){\makebox(0,0){$L$}}
\put(3,12){\vector(1,0){7}}
\put(9,0){\vector(-1,0){7}}
\put(12,2){\vector(0,1){8}}
\put(0,10){\vector(0,-1){8}}
\put(10,2){\vector(-1,1){8}}
\put(6.25,13.5){\makebox(0,0){$\ast$}}
\put(5.5,-2){\makebox(0,0){$a$}}
\put(14,6){\makebox(0,0){$0$}}
\put(-2,6){\makebox(0,0){$\pi$}}
\put(7.5,7.5){\makebox(0,0){$\widehat{a}$}}
\end{picture}
\end{center}
\end{figure}

\noindent Then $\widehat{a}\in\mathcal{R}_{0}2L$, and the hat map
$a\longmapsto\widehat{a}$ effects an $\mathbf{AT}$-isomorphism $\mathcal{R}%
L\rightarrow\mathcal{R}_{0}2L$.
\end{proposition}

\subsection{Filtered frames}

The standard representation $\tau_{M}:M\rightarrow2_{F}Q$ of a pointed frame
$M$ suggests that it may be helpful to think of $M$ in terms of $Q$ and $F$.
We are thus led to consider the category of frames with filters.

A \emph{filtered frame} is an object of the form $\left(  L,F\right)  $, where
$L$ is a frame and $F$ is a regular filter on $L$. A \emph{filtered frame
morphism} $f:\left(  L,F\right)  \rightarrow\left(  M,K\right)  $ is a pair
$\left(  c,f\right)  $, where $c\in M$ and $f:L\rightarrow{}\downarrow\!c$ is
a frame morphism such that
\[
a\in F\Longrightarrow c\rightarrow f\left(  a\right)  \in K,\;a\in L.
\]
Here $c\rightarrow f\left(  a\right)  =\bigvee_{b\wedge f\left(  a\right)
\leq c}b$ is the Heyting arrow operation .

\begin{lemma}
\label{Lem:28}In any frame,%
\[
a\leq b\leq c\Longrightarrow c\rightarrow\left(  b\rightarrow a\right)  \leq
b\rightarrow a.
\]

\end{lemma}

\begin{proof}
The conclusion follows from the fact that
\[
b\wedge\left(  c\rightarrow\left(  b\rightarrow a\right)  \right)  \leq
c\wedge\left(  c\rightarrow\left(  b\rightarrow a\right)  \right)  \leq
b\rightarrow a,
\]
hence
\[
b\wedge\left(  c\rightarrow\left(  b\rightarrow a\right)  \right)  \leq
b\wedge\left(  b\rightarrow a\right)  \leq a. \qedhere
\]

\end{proof}

\begin{lemma}
Let $\left(  c,f\right)  :\left(  L,F\right)  \rightarrow\left(  M,K\right)  $
and $\left(  d,g\right)  :\left(  M,K\right)  \rightarrow\left(  N,G\right)  $
be filtered frame morphisms. Then $\left(  g\left(  c\right)  ,gf\right)
:\left(  L,F\right)  \rightarrow\left(  N,G\right)  $ is a filtered frame morphism.
\end{lemma}

\begin{proof}
Clearly $gf$ is a frame morphism $L\rightarrow{}\downarrow\!g\left(  c\right)
$. Now%
\[
a\in F\Longrightarrow c\rightarrow f\left(  a\right)  \in K\Longrightarrow
d\rightarrow g\left(  c\rightarrow f\left(  a\right)  \right)  \in G,
\]
and%
\[
d\rightarrow g\left(  c\rightarrow f\left(  a\right)  \right)  \leq
d\rightarrow\left(  g\left(  c\right)  \rightarrow gf\left(  a\right)
\right)  \leq g\left(  c\right)  \rightarrow gf\left(  a\right)  .
\]
The first inequality is an instance of the rule that $h\left(  u\rightarrow
v\right)  \leq h\left(  u\right)  \rightarrow h\left(  v\right)  $ for any
frame map $h$, and the second inequality is an instance of Lemma \ref{Lem:28}.
\end{proof}

We denote the category of filtered frames with their morphisms by
$\mathbf{fF}$. Of special importance is the \emph{filtered frame of the reals}
$\left(  \mathcal{O}\left(  \mathbb{R}\smallsetminus\left\{  0\right\}
\right)  ,F_{0}\right)  $, where
\[
F_{0}\equiv\left\{  U\smallsetminus\left\{  0\right\}  :0\in U\in
\mathcal{O}\mathbb{R}\right\}  ,
\]
the filter of punctured neighborhoods of $0$.

A \emph{fully filtered frame} is a filtered frame $\left(  L,F\right)  $ for
which the filter $F$ is improper, i.e., a filtered frame of the form $\left(
L,L\right)  $. We denote the corresponding full subcategory by $\mathbf{ffF}$.
The point of the next lemma is that

\begin{lemma}
\label{Lem:27}The fully filtered frames comprise a full birecflective
subcategory of $\mathbf{fF}$. A reflector for $\left(  L,F\right)  $ is
\[
\left(  L,F\right)  \rightarrow\left(  L,L\right)  =\left(  a\longmapsto
a\right)  .
\]

\end{lemma}

\begin{proof}
It is clear that any morphism $\left(  L,F\right)  \rightarrow\left(
K,K\right)  $ factors through this map.
\end{proof}

\subsection{Pointed frames are categorically equivalent to filtered frames}

\begin{lemma}
In any frame,
\[
\left(  b\vee c = \top \text{ and } b \wedge c=d \right)  \Longrightarrow b =
c\rightarrow d.
\]

\end{lemma}

Let $\mathcal{D}:\mathbf{fF}\rightarrow\mathbf{pF}$ be the functor whose
action on an object is $\left(  L,F\right)  \longmapsto2_{F}L$ and whose
action on a morphism is
\[
\left(  a\longmapsto f\left(  a\right)  \right)  \longmapsto\left(  \left(
\varepsilon,a\right)  \longmapsto\left(  \varepsilon,f\left(  a\right)
\right)  \right)  ,
\]
and let $\mathcal{E}:\mathbf{pF}\rightarrow\mathbf{fF}$ be the functor whose
action on an object is
\[
\left(  M,\ast_{M}\right)  \longmapsto
     \left(  {\downarrow} p_M, p_M \wedge (M \smallsetminus {\downarrow} p_M)\right)
\]
and whose action on a morphism is
\[
\left(  g:\left(  M,\ast_{M}\right)  \rightarrow\left(  N,\ast_{N}\right)\right)  
\longmapsto g|_{L}, \qquad L\equiv{\downarrow} p_M .
\]

\begin{proposition}
\label{Prop:1}The functors
\[
\mathbf{fF}\underset{\mathcal{E}}{\overset{\mathcal{D}}{\leftrightarrows}%
}\mathbf{pF}%
\]
constitute a categorical equivalence. In particular, the restrictions of these
functors provide a categorical equivalence between $\mathbf{ffF}$ and
$\mathbf{ipF}$. The units of the equivalence are the isomorphisms
\begin{align*}
M  &  \rightarrow2_{F}L,\;L\equiv\ast_{M}^{-1}\left(  \bot\right)
,\;p\equiv\bigvee L,\;F\equiv p\wedge\left(  M\smallsetminus L\right)
,\;\left(  M,\ast_{M}\right)  \in\mathbf{pF},\\
\left(  L,F\right)   &  \rightarrow\left(  \bot\times L,\bot\times F\right)
,\;\left(  L,F\right)  \in\mathbf{fF},
\end{align*}
defined by the rules%
\begin{align*}
a  &  \longmapsto\left(  \ast_{M}\left(  a\right)  ,p\wedge a\right)
,\;p=\bigvee\ast_{M}^{-1}\left(  \bot\right)  ,\;a\in M,\\
a  &  \longmapsto\left(  \bot,a\right)  ,\;a\in L\text{.}%
\end{align*}

\end{proposition}

\begin{proof}
This is a straightforward elaboration of Proposition \ref{Prop:2} and its proof.
\end{proof}

It follows from Lemma \ref{Lem:27} and Proposition \ref{Prop:1} that
$\mathbf{ipF}$ is a full bireflective subcategory of $\mathbf{pF}$. Let us
explicitly record the reflector arrow.

\begin{corollary}
\label{Cor:4}The extension
\[
\nu_{M,}:\left(  M,\ast_{M}\right)  \rightarrow2L,\;L=\ast_{M}^{-1}\left(
\bot\right)  ,
\]
is the free isolated point frame over the pointed frame $\left(  M,\ast
_{M}\right)  $
\end{corollary}

\begin{proof}
We pointed out in Proposition \ref{Prop:2} that $\nu_{M}$ factors through
$2_{F}L$, and that the initial factor $M\rightarrow2_{F}L$ is a $\mathbf{pF}%
$-isomorphism. The final factor $2_{F}L\rightarrow2L$ is the extension of
Lemma \ref{Lem:27}.
\end{proof}

We offer an example for the reader's edification.

\begin{example}
Let $u$ be a free ultrafilter on a complete atomless Boolean algebra $B$, and
let $M$ be the pointed frame
\[
\left(  2_{u}B,\ast\right)  =\mathcal{D}\left(  B,u\right)  =\left\{  \left(
\varepsilon,b\right)  \in2\times B:\varepsilon=\top\Longrightarrow b\in
u\right\}  .
\]
Observe that the only point of $M$ is the designated point $\ast$, and it is
far from isolated. In fact, the free isolated point frame over $M$ is the
inclusion $2_{u}B\rightarrow2B=2\times B$. The only point of $2B$ is again its
designated point, but this time it is isolated. Note that the passage from
$2_{u}B$ to $2B$ represents a considerable enlargement of the frame.
\end{example}

\subsection{The trunc $\mathcal{R}_{\mathbf{fF}}\left(  L,F\right)  $}

For a filtered frame $\left(  L,F\right)  $, we denote the family of filtered
frame morphisms $\left(  \mathcal{O}\mathbb{R},F_{0}\right)  \rightarrow
\left(  L,F\right)  $ by
\begin{align*}
\mathcal{R}_{\mathbf{fF}}\left(  L,F\right)   &  \equiv\hom_{\mathbf{fF}%
}\left(  \left(  \mathcal{O}\mathbb{R},F_{0}\right)  ,\left(  L,F\right)
\right) \\
&  =\left\{  a\in\mathcal{R}L:\forall~U\in\mathcal{O}\mathbb{R}~\left(  0\in
U\Longrightarrow a\left(  U\right)  \in F\right)  \right\}  .
\end{align*}
It is easy to see that $\mathcal{R}_{\mathbf{fF}}\left(  L,F\right)  $ is
isomorphic to a subtrunc of $\mathcal{R}L$.

\begin{corollary}
\label{Cor:1}For a pointed frame $M$ with $\mathcal{E}M\equiv\left(
L,F\right)  $, $\mathcal{R}_{\mathbf{pF}}M$ and $\mathcal{R}_{\mathbf{fF}%
}\left(  L,F\right)  $ are isomorphic truncs.
\end{corollary}

\begin{proof}
In the diagram below, \begin{figure}[tbh]
\setlength{\unitlength}{4pt}
\par
\begin{center}
\begin{picture}(36,12)(-12,0)
\small
\put(0,12){\makebox(0,0){$2L$}}
\put(12,0){\makebox(0,0){$\mathcal{O}\mathbb{R}$}}
\put(12,12){\makebox(0,0){$2$}}
\put(0,0){\makebox(0,0){$L$}}
\put(-12,12){\makebox(0,0){$2_F L$}}
\put(-24,12){\makebox(0,0){$M$}}
\put(3,12){\vector(1,0){7}}
\put(9,0){\vector(-1,0){7}}
\put(12,2){\vector(0,1){8}}
\put(0,10){\vector(0,-1){8}}
\put(10,2){\vector(-1,1){8}}
\put(-8.5,12){\vector(1,0){6}}
\put(-22,12){\vector(1,0){6.5}}
\put(6.25,13.5){\makebox(0,0){$\ast$}}
\put(5.5,-2){\makebox(0,0){$a$}}
\put(14,6){\makebox(0,0){$0$}}
\put(-2,6){\makebox(0,0){$\pi$}}
\put(7.5,7.5){\makebox(0,0){$\widehat{a}$}}
\end{picture}
\end{center}
\end{figure}we may identify $M$ with $2_{F}L$ by Proposition \ref{Prop:2}.
Then the condition that $\widehat{a}$ factors through the insertion
$2_{F}L\rightarrow2L$ is exactly the condition that $a$ belongs to
$\mathcal{R}_{\mathbf{fF}}\left(  L,F\right)  $.
\end{proof}

We call an element $a\in L$ \emph{cocompact} if, for $S\subseteq L$,
\[
a\vee\bigvee S=\top\Longrightarrow\exists~S_{0}\subseteq_{\omega}S~\left(
a\vee\bigvee S_{0}=\top\right)  .
\]
Here the notation $S_{0}\subseteq_{\omega}S$ means that $S_{0}$ is a finite
subset of $S$.

Recall that the compactness degree of a frame $L$ is the least regular
cardinal $\kappa$ such that every subset $S\subseteq L$ such that $\bigvee
S=\top$ has a subset $S_{0}\subseteq_{\kappa}S$ with $\bigvee S_{0}=\top$. We
write $\operatorname{comp}L=\kappa$.

\begin{lemma}
\label{Lem:14}Let $L$ be a frame, let $F$ be a filter on $L$, and let $M$
abbreviate $2_{F}L$.

\begin{enumerate}
\item When restricted to $M$, the projection $M\rightarrow L\equiv\left(
\left(  \varepsilon,a\right)  \longmapsto a\right)  $ is dense iff $F$ is a
proper filter on $L$.

\item Suppose $\bigvee_{F}b^{\ast}=\top$. Then $M$ is regular if $L$ is.

\item $\operatorname{comp}M\leq\operatorname{comp}L$.

\item If $F$ is contained in the filter of cocompact elements of $L$ then $M$
is compact.
\end{enumerate}
\end{lemma}

\begin{proof}
(2) We claim that, for any $c\in L$,
\[
\left(  \bot,c\right)  =\bigvee\left\{  \left(  \bot,a\wedge b^{\ast}\right)
:b\in F,~a\prec c\right\}  .
\]
This is true because
\begin{gather*}
\bigvee_{b\in F,~a\prec c}\left(  \bot,a\wedge b^{\ast}\right)  = \left(
\bot,\bigvee_{b\in F,a\prec c}\left(  a\wedge b^{\ast}\right)  \right)  =
\left(  \bot,\bigvee_{a\prec c}\bigvee_{b\in F}\left(  a\wedge b^{\ast
}\right)  \right)  =\\
\left(  \bot,\bigvee_{a\prec c}\left(  a\wedge\bigvee_{b\in F}b^{\ast}\right)
\right)  =\left(  \bot,\bigvee_{a\prec c}a\right)  =\left(  \bot,c\right)  .
\end{gather*}
And $\left(  \top,a^{\ast}\vee b\right)  $ witnesses $\left(  \bot,a\wedge
b^{\ast}\right)  \prec\left(  \bot,c\right)  $, for
\[
\left(  \top,a^{\ast}\vee b\right)  \wedge\left(  \bot,a\wedge b^{\ast
}\right)  =\left(  \bot,\bot\right)  \;\text{and\ }\left(  \top,a^{\ast}\vee
b\right)  \vee\left(  \bot,c\right)  =\left(  \top,\top\right)  .
\]
On the other hand, it is obvious that $\left(  \top,b\right)  =\left(
\top,\bigvee_{a\prec b}a\right)  =\bigvee_{a\prec b}\left(  \top,a\right)  $
for any $b\in F$, and if $a\prec b$ it is just as clear that $\left(
\top,a\right)  \prec\left(  \top,b\right)  $. We leave the straightforward
proofs of (3) and (4) to the reader.
\end{proof}

\subsection{The spectrum of $A$}

We are finally prepared to introduce the frame canonically associated with $A$
in the functorial representation we seek. The \emph{spectral frame of }$A$ is
the frame
\[
\mathcal{M}A\equiv2_{F}\mathcal{K}A,
\]
where $F$ is the filter on $\mathcal{K}A$ generated by the truncation kernels
of the form $a\blacktriangleleft1$, $a\in\overline{A}$. Our use of the letter
$\mathcal{M}$ to denote the spectrum is intended to acknowledge the
contributions of James Madden, who was responsible in large part for the
localic representation in $\mathbf{W}$ (\cite{Madden:1990}). We abbreviate
$\mathcal{M}A$ to $M$ for the rest of this section.

\begin{theorem}
\label{Thm:2}$M$ is a regular Lindel\"{o}f frame.
\end{theorem}

\begin{proof}
Lemma \ref{Lem:14} is relevant here. By part (3), and in light of Theorem
\ref{Thm:1}, $M$ is Lindel\"{o}f. By part (2), we need only show that
\[
\bigvee_{\overline{A}}\left(  \bigvee_{0<r<1}\left[  \left(  a\ominus
r\right)  \right]  ^{\ast}\right)  ^{\ast}=\top\text{ \ in }\mathcal{K}A
\]
in order to show that $M$ is regular. Since the identity $\left(  \bigvee
a^{\ast}\right)  ^{\ast}=\bigwedge a^{\ast\ast}$ holds in any frame, this
amounts to showing that
\begin{equation}
\bigvee_{\overline{A}}\bigwedge\limits_{0\leq r<1}\left[  \left(  a\ominus
r\right)  \right]  ^{\ast\ast}=\top. \tag{$\ast$}%
\end{equation}
For that purpose fix $a\in\overline{A}$, $r\in\mathbb{Q}$, and $n\in
\mathbb{N}$ such that $r<1$. From Lemma \ref{Lem:9} we get
\[
\left[  na\ominus1\right]  \prec\left[  na\ominus r\right]  \Longrightarrow
\left[  na\ominus1\right]  \leq\left[  na\ominus r\right]  ^{\ast\ast},
\]
and, by letting the $r$ vary, $\left[  na\ominus1\right]  \leq\bigwedge_{0\leq
r<1}\left[  \left(  a\ominus r\right)  \right]  ^{\ast\ast}$. Thus, whatever
truncation kernel is represented by the left side of the expression in ($\ast
$), it contains $na\ominus1$ for all $n$. But then it contains $a$ by Lemma
\ref{Lem:1}, and because $a$ was chosen arbitrarily, it contains all the
elements of $\overline{A}$. Since it must satisfy part (2) of Lemma
\ref{Lem:5}, it must contain $A$. We have shown $M$ to be regular.
\end{proof}

Our plan is to represent $A$ as a subobject of $\mathcal{R}_{0}\mathcal{M}A$.
One important detail remains to be checked.

\begin{lemma}
$\underline{A}$ is a subtrunc of $\mathcal{R}_{\mathbf{fF}}\mathcal{K}A$.
\end{lemma}

\begin{proof}
We must show that $\underline{a}\left(  U\right)  \in F$ for $a\in A$ and
$U\in\mathcal{O}\mathbb{R}$ such that $0\in U$. Without loss of generality we
may assume that $U$ has the form $\left(  -\varepsilon,\varepsilon\right)  $
for some $0<\varepsilon\in\mathbb{R}$, and, since $\underline{a}\left(
-\varepsilon,\varepsilon\right)  =\underline{\left\vert a\right\vert }\left(
-\infty,\varepsilon\right)  $, we need only show that $\underline{a}\left(
-\infty,\varepsilon\right)  \in F$ for any $a\in A^{+}$. According to Lemma
\ref{Lem:17}, $\underline{a}\left(  -\infty,\varepsilon\right)  =\bigvee
_{0<s<\varepsilon}\left[  a\ominus s\right]  ^{\ast}$. But if we replace $s$
by $r\varepsilon$, we get
\[
\underline{a}\left(  -\infty,\varepsilon\right)  =\bigvee_{0<r<1}\left[
a\ominus r\varepsilon\right]  ^{\ast}=\bigvee_{0<r<1}\left[  \overline
{a/\varepsilon}\ominus r\right]  ^{\ast}=\overline{a/\varepsilon
}\blacktriangleleft1\in F.
\]
The second equality is justified by the observation that
\[
\left[  a\ominus r\varepsilon\right]  =\left[  \varepsilon\left(
a/\varepsilon\ominus r\right)  \right]  =\left[  a/\varepsilon\ominus
r\right]  =\left[  \overline{a/\varepsilon}\ominus r\right]  . \qedhere
\]

\end{proof}

\section{The functorial representation\label{Sec:2}}

We have in hand the components of the representation we seek.
\begin{figure}[tbh]
\setlength{\unitlength}{4pt}
\par
\begin{center}
\begin{picture}(25,1)(0,0)
\put(-21,0){\makebox(0,0){$A$}}
\put(0,0){\makebox(0,0){$\underline{A} \leq \mathcal{R}_F \mathcal{K} A$}}
\put(26,0){\makebox(0,0){$\mathcal{R}_0 \mathcal{M} A$}}
\put(-19,0){\vector(1,0){10}}
\put(9,0){\vector(1,0){11}}
\put(-14.5,1.25){\makebox(0,0){\small $a \mapsto \underline{a}$}}
\put(14.5,1.5){\makebox(0,0){\small $\underline{a} \mapsto \widehat{a}$}}
\end{picture}
\end{center}
\end{figure}

\noindent Combining these components results in Theorem \ref{Thm:4}, which
summarizes the development to this point.

\subsection{The representation of objects}

\begin{theorem}
\label{Thm:4}For $a\in A^{+}$, define
\[
\widehat{a}\left(  r,\infty\right)  \equiv\left\{
\begin{array}
[c]{ll}%
\left(  \bot,a\blacktriangleright r\right)  , & r\geq0\\
\left(  \top,\top\right)  , & r<0
\end{array}
\right.  .
\]
Then $\widehat{a}$ extends to a unique element $\widehat{a}\in\mathcal{R}%
_{p}\mathcal{M}A$, and the map $a\longmapsto\widehat{a}$, $a\in A^{+}$,
extends to a unique truncation isomorphism
\[
\mu_{A}:A\rightarrow\widehat{A}\equiv\left\{  \widehat{a}:a\in A\right\}
\leq\mathcal{R}_{p}\mathcal{M}A.
\]

\end{theorem}

We turn now to the issue of functoriality. This requires a few pertinent facts
about subtruncs of $\mathcal{R}L$, for $L$ a frame.

\subsection{Cozero facts}

\emph{Throughout this subsection }$L$\emph{\ designates a frame and }%
$A$\emph{\ designates a subtrunc of }$\mathcal{R}L$. Let us recall some
standard terminology relevant to this situation. The \emph{cozero element of
}$a\in A$ is
\[
\operatorname{coz}a\equiv a\left(  \left(  -\infty,0\right)  \cup\left(
0,\infty\right)  \right)  .
\]
In similar spirit we define the \emph{co-one element of }$a$ to be
\[
\operatorname{con}a\equiv a\left(  \left(  -\infty,1\right)  \cup\left(
1,\infty\right)  \right)  .
\]
If $a\in\overline{A}$ then these expressions simplify to $\operatorname{coz}%
a=a\left(  0,\infty\right)  $ and $\operatorname{con}a=\left(  -\infty
,1\right)  $. Finally, we will frequently and without comment use the fact
that $\operatorname{con}a=\bigvee_{s<1}a\left(  s,\infty\right)  ^{\ast}$.

\begin{lemma}
For any subset $S\subseteq A$, $\bigvee_{S}\operatorname{coz}a=\bigvee
_{\left[  S\right]  }\operatorname{coz}a$.
\end{lemma}

\begin{proof}
By Lemma \ref{Lem:24}, we need only show that $\bigvee_{S}\operatorname{coz}%
a=\bigvee_{S^{\alpha}}\operatorname{coz}a$ for all $\alpha$, and this we do by
induction. If $\alpha=0$ then $S^{\alpha}=\left\langle S\right\rangle $, the
convex $\ell$-subgroup of $A$ generated by $S$. Keeping in mind the facts that
$\operatorname{coz}\alpha=\operatorname{coz}\left\vert a\right\vert $, $0\leq
a\leq b$ implies $\operatorname{coz}a\leq\operatorname{coz}b$, and
$\operatorname{coz}\left\vert a\vee b\right\vert $, $\operatorname{coz}%
\left\vert a+b\right\vert \leq\operatorname{coz}\left\vert a\right\vert
\vee\operatorname{coz}\left\vert b\right\vert $, the truth of the assertion is
clear in this case. Assume now that the assertion holds for all $\gamma
<\alpha$. If $\alpha$ is a limit ordinal then the assertion clearly also holds
at $\alpha$, so assume $\alpha$ is of the form $\gamma+1$. If $\gamma
\equiv0\operatorname{mod}3$ and $b\in S^{\alpha+}$ then there is some $c\in
S^{\gamma+}$ for which $\left(  nb-c\right)  ^{+}\in S^{\gamma}$ for all $n$.
But since $\operatorname{coz}b=\bigvee_{\mathbb{N}}\operatorname{coz}\left(
nb-c\right)  ^{+}$, it follows that $\operatorname{coz}b\leq\bigvee
_{S^{\gamma}}\operatorname{coz}a$, with the result that
\[
\bigvee_{S^{\alpha}}\operatorname{coz}a=\bigvee_{S^{\gamma}}\operatorname{coz}%
a=\bigvee_{S}\operatorname{coz}a.
\]
The argument for the case in which $\gamma\equiv1\operatorname{mod}3$ is
almost identical, and, in light of the fact that $\operatorname{coz}%
\overline{a}=\operatorname{coz}a $ for $a\in A^{+}$, the argument for the case
in which $\gamma\equiv2\operatorname{mod}3$ is trivial.
\end{proof}

\begin{lemma}
\label{Lem:20}For $a,b\in A^{+}$ with $a\in\overline{A}$ and $b\in
a\blacktriangleleft1$, $\operatorname{coz}b\leq\operatorname{con}a.$ Therefore
$\operatorname{con}a\geq\bigvee_{a\blacktriangleleft1}\operatorname{coz}b$.
\end{lemma}

\begin{proof}
Abbreviate $\bigcup_{\mathbb{N}}a\ominus\left(  1-1/n\right)  ^{\bot}$ to $K$.
Since
\[
a\blacktriangleleft1\equiv\bigvee_{\mathbb{N}}\left[  a\ominus\left(
1-\frac{1}{n}\right)  \right]  ^{\ast}=\bigvee_{\mathbb{N}}a\ominus\left(
1-\frac{1}{n}\right)  ^{\bot}=\left[  K\right]  =K^{\omega_{1}}%
\]
by Propositions \ref{Prop:1} and \ref{Prop:3}, it is sufficient to demonstrate
that, for all $\alpha$, $\operatorname{coz}b\leq\operatorname{con}a\,$whenever
$b\in K^{\alpha+}$. This we do by induction on $\alpha$. If $\alpha=0$ then we
would have $b\wedge a\ominus\left(  1-1/n\right)  =0$ for some $n$, with the
result that
\begin{gather*}
\bot=\operatorname{coz}0=\operatorname{coz}\left(  b\wedge a\ominus\left(
1-\frac{1}{n}\right)  \right)  =\\
\operatorname{coz}b\wedge\operatorname{coz}\left(  a-\left(  1-\frac{1}%
{n}\right)  \right)  ^{+} =\operatorname{coz}b\wedge a\left(  1-\frac{1}%
{n},\infty\right)  .
\end{gather*}
Therefore
\[
\top=\operatorname{con}a\vee a\left(  1-\frac{1}{n},\infty\right)
\Longrightarrow\operatorname{coz}b\leq\operatorname{con}a.
\]
Now assume the assertion holds for all $\gamma$, $\gamma<\alpha<\omega_{1}$.
If $\alpha$ is a limit ordinal then the assertion holds also at $\alpha$, so
assume $\alpha=\gamma+1$ for some $\gamma$. If $\gamma\equiv
0\operatorname{mod}3$ then there is some $c\in K^{\gamma+}$ such that $\left(
nb-c\right)  ^{+}\in K^{\gamma}$ for all $n$. By the inductive hypothesis we
have $\operatorname{coz}\left(  nb-c\right)  ^{+}\leq\operatorname{con}a$ for
all $n$. Now
\[
\operatorname{coz}\left(  nb-c\right)  ^{+}=\left(  nb-c\right)  \left(
0,\infty\right)  =\bigvee_{nU_{1}-U_{2}\subseteq\left(  0,\infty\right)
}\left(  b\left(  U_{1}\right)  \wedge c\left(  U_{2}\right)  \right)
\]
But if open subsets $U_{i}\subseteq\mathbb{R}$ satisfy $nU_{1}-U_{2}%
\subseteq\left(  0,\infty\right)  $ then $U_{1}$ must be bounded below, say by
$t_{1}$, and $U_{2}$ must be bounded above, say by $t_{2}$, where
$nt_{1}-t_{2}>0$. That is to say that, in the last supremum displayed above,
$U_{1}$ and $U_{2}$ may be replaced by $\left(  t,\infty\right)  $ and
$\left(  -\infty,nt\right)  $ for some $t\in\mathbb{R}$. This gives%
\begin{gather*}
\operatorname{con}a \geq\bigvee_{n}\bigvee_{t}\left(  b\left(  t,\infty
\right)  \wedge c\left(  -\infty,nt\right)  \right)  =\\
\bigvee_{t}\bigvee_{n}\left(  b\left(  t,\infty\right)  \wedge c\left(
-\infty,nt\right)  \right)  =\bigvee_{t}\left(  b\left(  t,\infty\right)
\wedge\bigvee_{n}c\left(  -\infty,nt\right)  \right)  .
\end{gather*}
But $\bigvee_{n}c\left(  -\infty,nt\right)  =\bot$ for $t\leq0$ since $c\geq0$
and $0\left(  -\infty,0\right)  =\bot$, while $\bigvee_{n}c\left(
-\infty,nt\right)  =\top$ for $t>0$. Therefore the last expression displayed
above works out to $\bigvee_{t>0}b\left(  t,\infty\right)  =b\left(
0,\infty\right)  =\operatorname{coz}b$, as desired.

Consider next the case in which $\gamma\equiv1\operatorname{mod}3$. By the
inductive hypothesis we have, for each $n\in\mathbb{N}$,
\begin{align*}
\operatorname{con}a  &  \geq\operatorname{coz}nb\ominus1=\operatorname{coz}%
\left(  nb-1\right)  ^{+}=\left(  nb-1\right)  \left(  0,\infty\right) \\
&  =nb\left(  1,\infty\right)  =b\left(  \frac{1}{n},\infty\right)  ,
\end{align*}
with the result that $b\left(  0,\infty\right)  =\bigvee_{n}b\left(
1/n,\infty\right)  \leq\operatorname{con}a$. In the last case $\gamma
\equiv2\operatorname{mod}3$, and $\operatorname{coz}\overline{b}%
\leq\operatorname{con}a$ by the inductive hypothesis. But
\[
\operatorname{coz}\overline{b}=\left(  b\wedge1\right)  \left(  0,\infty
\right)  =b\left(  0,\infty\right)  =\operatorname{coz}b,
\]
and so the proof is complete.
\end{proof}

\begin{corollary}
For any $K\in\mathcal{K}A$, $\bigvee_{^{1}K}\operatorname{con}a\geq
\bigvee_{^{0}K}\operatorname{coz}b=\bigvee_{K}\operatorname{coz}b$.
\end{corollary}

\begin{proof}
We have%
\[
\bigvee_{^{1}K}\operatorname{con}a\geq\bigvee_{a\in{}^{1}K}\bigvee_{b\in
a\blacktriangleleft1}\operatorname{coz}b=\bigvee_{K}\operatorname{coz}b.
\]
The equality holds because $a\in{}^{1}K$ means that $a\blacktriangleleft
1\subseteq K$, and because $K=\bigcup_{^{1}K}a\blacktriangleleft1$.
\end{proof}

\begin{lemma}
\label{Lem:21}For $a,b\in\overline{A}$ with $a\in\overline{A}$,
$\operatorname{coz}b\wedge\operatorname{con}a\leq\bigvee_{a\blacktriangleleft
1}\operatorname{coz}c$.
\end{lemma}

\begin{proof}
Because
\begin{gather*}
\operatorname{coz}b =\operatorname{coz}b\wedge\top=\operatorname{coz}%
b\wedge\operatorname{coz}\frac{1}{2} =\operatorname{coz}\left(  b\wedge
\frac{1}{2}\right)  =\operatorname{coz}\left(  \frac{1}{2}\overline
{2b}\right)  ,
\end{gather*}
we may assume that $b=\left(  1/2\right)  \overline{2b}$. Since
$\operatorname{con}a=\bigvee_{\mathbb{N}}a\left(  1-1/n,\infty\right)  ^{\ast
}$, it is sufficient to show that for each $n\in\mathbb{N}$ there exists $c\in
a\blacktriangleleft1^{+}$ such that
\[
\operatorname{coz}b\wedge a\left(  1-1/n,\infty\right)  ^{\ast}\leq
\operatorname{coz}c.
\]
Fix $n$, and put $a_{1}\equiv na\ominus\left(  n-1\right)  $. Note that
$a_{1}\in\overline{A}$ because
\[
a\ominus1=na\ominus\left(  n-1\right)  \ominus1=na\ominus na=n\left(
a\ominus1\right)  =0.
\]
Let $c\equiv\left(  b-a_{1}\right)  ^{+}$, and observe that
\begin{gather*}
c\wedge\left(  a_{1}-b\right)  ^{+} =\left(  b-a_{1}\right)  ^{+}\wedge\left(
a_{1}-b\right)  ^{+}=0,\tag*{$\ast$}%
\end{gather*}
and $2\left(  a_{1}\vee c\right)  \geq b$. By Corollary \ref{Cor:2}(3),
\begin{gather*}
\left(  a_{1}-b\right)  ^{+} \geq a_{1}\ominus\frac{1}{2}= na\ominus\left(
n-1\right)  \ominus\frac{1}{2}=\\
na\ominus\left(  na-\frac{1}{2}\right)  =n\left(  a\ominus\left(  1-\frac
{1}{2n}\right)  \right)  .
\end{gather*}
Combined with ($\ast$), this yields $c\wedge n\left(  a\ominus\left(
1-\frac{1}{2n}\right)  \right)  =0$, hence $c\wedge a\ominus\left(  1-\frac
{1}{2n}\right)  =0$, i.e., $c\in a\ominus\left(  1-1/2\right)  ^{\bot
}\subseteq a\blacktriangleleft1.$ With the aid of Lemma \ref{Lem:11} we now
get%
\begin{gather*}
\operatorname{coz}a_{1} =\operatorname{coz}na\ominus\left(  n-1\right)
=\operatorname{coz}\left(  na-\left(  n-1\right)  \right)  ^{+} =\left(
na\right)  \left(  n-1,\infty\right)  =a\left(  1-\frac{1}{n},\infty\right)
\end{gather*}
From the inequality in ($\ast$) comes the information that
\begin{gather*}
\operatorname{coz}b \leq\operatorname{coz}2\left(  a_{1}\vee c\right)
=\operatorname{coz}\left(  a_{1}\vee c\right) \\
=\operatorname{coz}a_{1}\vee\operatorname{coz}c=a\left(  1-\frac{1}{n}%
,\infty\right)  \vee\operatorname{coz}c.
\end{gather*}
If we now meet both sides with $a\left(  1-1/n\right)  ^{\ast}$ we get
\[
\operatorname{coz}b\wedge a\left(  1-\frac{1}{n}\right)  ^{\ast}%
=\operatorname{coz}c\wedge a\left(  1-1/n\right)  ^{\ast}\leq
\operatorname{coz}c. \qedhere
\]
\end{proof}

\begin{proposition}
\label{Prop:4}For $a,b\in A^{+}$ with $a\in\overline{A}$,
\[
\operatorname{coz}b\wedge\operatorname{con}a=\operatorname{coz}b\wedge
\bigvee_{a\blacktriangleleft1}\operatorname{coz}c=\bigvee_{c\in
a\blacktriangleleft1}\operatorname{coz}\left(  b\wedge c\right)  .
\]

\end{proposition}

\begin{proof}
According to Lemma \ref{Lem:20}, $\operatorname{con}a\geq\bigvee
_{a\blacktriangleleft1}\operatorname{coz}c$. And according to Lemma
\ref{Lem:21}, $\operatorname{coz}b\wedge\operatorname{con}a\leq\bigvee
_{a\blacktriangleleft1}\operatorname{coz}c.$ Together, these two facts imply
the equality asserted in the proposition.
\end{proof}

\begin{proposition}
\label{Prop:5}For $a,b\in\overline{A}$, $\operatorname{con}a\vee
\bigvee_{b\blacktriangleleft1}\operatorname{coz}c\geq\operatorname{con}b$.
\end{proposition}

\begin{proof}
We begin with the observation that, by Lemma \ref{Lem:22},
$b\blacktriangleleft1\vee b\blacktriangleright r=\top$ in $\mathcal{K}A$ for
$r<1$. Therefore, if we let $I\equiv b\blacktriangleleft1\cup\left\{  b\ominus
r\right\}  $, we can say that
\[
\operatorname{coz}b\ominus r\vee\bigvee_{b\blacktriangleleft1}%
\operatorname{coz}c=\bigvee_{C}\operatorname{coz}c=\bigvee_{\left[  C\right]
}\operatorname{coz}c=\bigvee_{A}\operatorname{coz}c.
\]
Consequently $\operatorname{coz}b\ominus r\vee\bigvee_{b\blacktriangleleft
1}\operatorname{coz}c\geq\operatorname{coz}a$, with the result that
\begin{gather*}
\operatorname{con}a\vee\operatorname{coz}b\ominus r\vee\bigvee
_{b\blacktriangleleft1}\operatorname{coz}c \geq\operatorname{coz}%
a\vee\operatorname{con} a =a\left(  0,\infty\right)  \vee a\left(
-\infty,1\right)  =\top.
\end{gather*}
But this implies that
\[
\operatorname{con}a\vee\bigvee_{b\blacktriangleleft1}\operatorname{coz}%
c\geq\left(  \operatorname{coz}b\ominus r\right)  ^{\ast}=b\left(
r,\infty\right)  ^{\ast}%
\]
for $r<1$, and, since $\operatorname{con}b=\bigvee_{r<1}b\left(
r,\infty\right)  ^{\ast}$, this proves the proposition.
\end{proof}

\subsection{The representation of morphisms}

We show that $\mathcal{R}_{p}$ is adjoint, which is to say that $\left(
\mu_{A},\mathcal{M}A\right)  $ is an $\mathcal{R}_{p}$-universal arrow with
domain $A$.

\begin{theorem}
\label{Thm:3} For any trunc morphism $\theta:A\rightarrow\mathcal{R}_{p}L$
there is a unique pointed frame morphism $g$ such that $\mathcal{R}_{p}%
g\circ\mu_{A}=\theta$. \begin{figure}[tbh]
\setlength{\unitlength}{4pt}
\par
\begin{center}
\begin{picture}(48,13)(0,1)
\small
\put(0,12){\makebox(0,0){$A$}}
\put(12,0){\makebox(0,0){$\mathcal{R}_0 L$}}
\put(12,12){\makebox(0,0){$\mathcal{R}_0\mathcal{M}A$}}
\put(24,12){\makebox(0,0){$\mathcal{M}A$}}
\put(24,0){\makebox(0,0){$L$}}
\put(48,6){\makebox(0,0){$\mathcal{O}_0\mathbb{R}$}}
\put(36,6){\makebox(0,0){$2$}}
\put(2,12){\vector(1,0){4.5}}
\put(2,10){\vector(1,-1){8}}
\put(12,10){\vector(0,-1){8}}
\put(24,10){\vector(0,-1){8}}
\put(28,10){\vector(2,-1){6}}
\put(28,2){\vector(2,1){6}}
\put(48,8){\line(0,1){4}}
\put(48,4){\line(0,-1){4}}
\put(48,12){\vector(-1,0){20.5}}
\put(48,0){\vector(-1,0){22}}
\put(44,6){\vector(-1,0){6}}
\put(15,6){\makebox(0,0){$\mathcal{R}_0g$}}
\put(22.5,6){\makebox(0,0){$g$}}
\put(4,14){\makebox(0,0){$\mu_A$}}
\put(3.5,5){\makebox(0,0){$\theta$}}
\put(36,14){\makebox(0,0){$\widehat{a}$}}
\put(36,-2.5){\makebox(0,0){$\theta (a)$}}
\end{picture}
\end{center}
\end{figure}
\end{theorem}

\begin{proof}
Observe that $g\left(  \bot,a\blacktriangleright0\right)  =g\widehat{a}\left(
0,\infty\right)  =\theta\left(  a\right)  \left(  0,\infty\right)
=\operatorname{coz}\theta\left(  a\right)  $ for any $a\in\overline{A}$, from
which it follows that
\begin{gather*}
g\left(  \bot,K\right)  = g\left(  \bot,\bigvee_{K^{0}}a\blacktriangleright
0\right)  = g\left(  \bigvee_{K^{0}}\left(  \bot,a\blacktriangleleft0\right)
\right)  =\\
\bigvee_{K^{0}}g\left(  \bot,a\blacktriangleleft0\right)  = \bigvee_{K^{0}%
}\theta\left(  a\right)  \left(  0,\infty\right)  = \bigvee_{K^{0}%
}\operatorname{coz}\theta\left(  a\right)  .
\end{gather*}
Likewise $g\left(  \top,a\blacktriangleleft1\right)  =g\widehat{a}\left(
-\infty,1\right)  =\theta\left(  a\right)  \left(  -\infty,1\right)
=\operatorname{con}\theta\left(  a\right)  $ for any $a\in\overline{A}$, from
which it follows that
\begin{gather*}
g\left(  \top,K\right)  = g\left(  \top,\bigvee_{K^{1}}a\blacktriangleleft
1\right)  = g\left(  \bigvee_{K^{1}}\left(  \top,a\blacktriangleleft1\right)
\right)  =\\
\bigvee_{K^{1}}g\left(  \top,a\blacktriangleleft1\right)  = \bigvee_{K^{1}%
}\theta\left(  a\right)  \left(  -\infty,1\right)  = \bigvee_{K^{1}%
}\operatorname{con}\theta\left(  a\right)  .
\end{gather*}
Therefore we have no choice but to define
\[
g\left(  \varepsilon,K\right)  \equiv\left\{
\begin{array}
[c]{cc}%
\bigvee_{K^{0}}\operatorname{coz}\theta\left(  a\right)  , & \varepsilon
=\bot\\
\bigvee_{K^{1}}\operatorname{con}\theta\left(  a\right)  , & \varepsilon=\top
\end{array}
\right.  ,\;K\in\mathcal{K}A,\varepsilon\in2.
\]
Clearly $g\left(  \bot,0\right)  =\bot$, and
\[
A=0\blacktriangleleft1\Longrightarrow0\in A^{1}\Longrightarrow g\left(
\top,A\right)  \geq0\left(  -\infty,1\right)  =\top.
\]
The proof is completed by showing that $g$ preserves binary meets and
arbitrary joins. This we do in a sequence of lemmas, all phrased in the
notation above.
\end{proof}

\begin{lemma}
$g$ preserves binary meets.
\end{lemma}

\begin{proof}
Consider $\left(  \varepsilon_{i},K_{i}\right)  \in M$. In the first case
$\varepsilon_{0}=\varepsilon_{1}=\bot$, so we have
\begin{gather*}
g\left(  \bot,K_{0}\right)  \wedge g\left(  \bot,K_{1}\right)  =
\bigvee_{K_{0}^{0}}\coz\theta\left(  a_{0}\right)  \wedge\bigvee_{K_{1}^{0}}
\coz\theta\left(  a_{1}\right)  = \bigvee_{a_{i}\in K_{i}^{0}}\left(
\coz\theta\left(  a_{0}\right)  \wedge\coz\theta\left(  a_{1}\right)  \right)
=\\
\bigvee_{K_{i}^{0}}\coz\theta\left(  a_{0}\wedge a_{1}\right)  =\bigvee
_{\left(  K_{0}\wedge K_{1}\right)  ^{0}}\coz\theta\left(  a\right)  =
g\left(  \bot,K_{0}\wedge K_{1}\right)  =g\left(  \left(  \bot,K_{0}\right)
\wedge\left(  \bot,K_{1}\right)  \right)  .
\end{gather*}
In the second case $\varepsilon_{0}=\varepsilon_{1}=\top$, and the argument
goes along similar lines. In the third and last case $\varepsilon_{0}%
=\bot<\top=\varepsilon_{1}$, so we have
\begin{gather*}
g\left(  \bot,K_{0}\right)  \wedge g\left(  \top,K_{1}\right)  =
\bigvee_{K_{0}^{0}}\coz\theta\left(  a_{0}\right)  \wedge\bigvee_{K_{1}^{1}%
}\cone\theta\left(  a_{1}\right)  = \bigvee_{a_{i}\in K_{i}^{i}}\left(
\coz\theta\left(  a_{0}\right)  \wedge\cone\theta\left(  a_{1}\right)
\right)
\end{gather*}
By Proposition \ref{Prop:4},%
\begin{gather*}
\bigvee_{a_{i}\in K_{i}^{i}}\left(  \coz\theta\left(  a_{0}\right)
\wedge\cone\theta\left(  a_{1}\right)  \right)  = \bigvee_{a_{i}\in K_{i}^{i}%
}\left(  \coz\theta\left(  a_{0}\right)  \wedge\bigvee_{c\in a_{1}%
\blacktriangleleft1}\coz\theta\left(  c\right)  \right)  =\\
\bigvee_{a_{i}\in K_{i}^{i}}\bigvee_{c\in a_{1} \blacktriangleleft1}\left(
\coz\theta\left(  a_{0}\right)  \wedge\coz\theta\left(  c\right) \right)  .
\end{gather*}
At this point it is useful to remind the reader that $K_{1}^{1}\equiv\left\{
a\in\overline{A}:a\blacktriangleleft1\subseteq K_{1}\right\}  $, so that
\begin{gather*}
\bigvee_{a_{i}\in K_{i}^{i}}\bigvee_{c\in a_{1}\blacktriangleleft1} \left(
\coz\theta\left(  a_{0}\right)  \wedge\coz\theta\left(  c\right)  \right)  =
\bigvee_{a_{i}\in K_{i}^{i}}\left(  \coz\theta\left(  a_{0}\right)
\wedge\coz\theta\left(  a_{1}\right)  \right)  =\\
\bigvee_{a\in\left(  K_{1}\wedge K_{2}\right)  ^{0}}\coz\theta\left( a\right)
= g\left(  \bot,K_{1}\wedge K_{2}\right)  = g\left(  \left(  \bot
,K_{1}\right)  \wedge\left(  \bot,K_{2}\right) \right) \qedhere
\end{gather*}
\end{proof}

\begin{lemma}
\label{Lem:25}$g$ preserves all joins of the form
\begin{align*}
\bigvee_{I}\left(  \bot,K_{i}\right)   &  =\left(  \bot,\bigvee_{I}%
K_{i}\right)  ,\;\left(  \bot,K_{i}\right)  \in\mathcal{M}A,\;\;\text{or}\\
\bigvee_{I}\left(  \top,K_{i}\right)   &  =\left(  \top,\bigvee_{I}%
K_{i}\right)  ,\;\left(  \top,K_{i}\right)  \in\mathcal{M}A.
\end{align*}

\end{lemma}

\begin{proof}
Let $K\equiv\bigvee_{I}K_{i}$. We have
\begin{gather*}
\bigvee_{I}g\left(  \bot,K_{i}\right)  = \bigvee_{I}\bigvee_{^{0}K_{i}%
}\coz\theta\left(  a_{i}\right)  = \bigvee_{I}\bigvee_{K_{i}} \coz\theta\left(
a_{i}\right)  = \bigvee_{\cup K_{i}}\coz\theta\left(  a\right)  =\\
\bigvee_{\left[  \cup K_{i}\right]  }\coz\theta\left(  a\right)  = \bigvee
_{K}\coz\theta\left( a\right)  = g\left(  \bot,K\right)  =g\left(  \bigvee
_{I}\left(  \bot,K_{i}\right) \right)
\end{gather*}
On the other hand,
\begin{gather*}
\bigvee_{I}g\left(  \top,K_{i}\right)  = \bigvee_{I}\bigvee_{^{1}K_{i}%
}\cone\theta\left(  a_{i}\right)  = \bigvee_{\cup{}^{1}K_{i}}\cone\theta\left(
a\right)  \leq\\
\bigvee_{{}^{1}K}\cone\theta\left(  a\right)  = g\left(  \top,K\right)
=g\left(  \bigvee_{I}\left(  \top,K_{i}\right) \right)  .
\end{gather*}
To prove the opposite inequality, observe first that, by Lemma \ref{Lem:20},
\begin{gather*}
\bigvee_{I}g\left(  \top,K_{i}\right)  = \bigvee_{I}\bigvee_{^{1}K_{i}}
\cone\theta\left(  a_{i}\right)  \geq\bigvee_{I}\bigvee_{K_{i}}\coz\theta
\left(  a_{i}\right)  =\\
\bigvee_{\cup K_{i}}\coz\theta\left(  a\right)  =\bigvee_{\left[  \cup
K_{i}\right]  }\coz\theta\left(  a\right)  =\bigvee_{K}\coz\theta\left(
a\right)  .
\end{gather*}
Fix $i_{0}\in I$ and $a_{0}\in{}^{1}K_{i_{0}}$, so that
\[
\bigvee_{I}g\left(  \top,K_{i}\right)  \geq g\left(  \top,K_{i_{0}}\right)
=\bigvee_{^{1}K_{i_{0}}}\cone\theta\left(  a\right)  \geq\cone\theta\left(
a_{0}\right)  .
\]
Now consider an arbitrary $c\in{}^{1}K$, for which we would have $\bigvee
_{K}\coz\theta\left(  a\right)  \geq\bigvee_{c\blacktriangleleft1}
\coz\theta\left(  a\right)  $ since $c\blacktriangleleft1\subseteq K$. We then
get from Proposition \ref{Prop:5} that
\[
\bigvee_{I}g\left(  \top,K_{i}\right)  \geq\cone\theta\left(  a_{0}\right)
\vee\bigvee_{K}\coz\theta\left(  a\right)  \geq\cone\theta\left(
a_{0}\right)  \vee\bigvee_{c\blacktriangleleft1}\coz \theta\left(  a\right)
\geq\cone c.
\]
Since $c$ was chosen arbitrarily, we have proven that $\bigvee_{I}g\left(
\top,K_{i}\right)  \geq\bigvee_{^{1}K}\cone$, which is to say that we have
proven the lemma.
\end{proof}

\begin{lemma}
\label{Lem:26}$g$ preserves binary joins of the form
\[
\left(  \bot,K_{0}\right)  \vee\left(  \top,K_{1}\right)  =\left(  \top,
K_{1}\vee K_{2}\right)  ,\;\left(  \bot,K_{0}\right)  , \left(  \top
,K_{1}\right)  \in\mathcal{M}A.
\]

\end{lemma}

\begin{proof}
To this join $g$ assigns the frame element $g\left(  \top,K\right)
=\bigvee_{^{1}K}\cone\theta\left(  a\right)  $, where $K\equiv K_{1}\vee
K_{2}$. We should compare this to
\begin{gather*}
g\left(  \bot,K_{0}\right)  \vee g\left(  \top,K_{1}\right)  = \left(
\bigvee_{^{0}K_{0}}\coz\theta\left(  b\right)  \right)  \vee\left(
\bigvee_{^{1}K_{1}}\cone\theta\left(  a\right)  \right)  =\\
\bigvee_{b\in{}^{0}K_{0}}\bigvee_{a\in{}^{1}K_{1}} \left(  \coz\theta\left(
b\right)  \vee\cone\theta\left(  a\right)  \right)  = \bigvee_{b\in{}^{0}%
K_{0}}\bigvee_{a\in{}^{1}K_{1}}\cone\theta\left(  a-b\right)  ^{+}.
\end{gather*}
Now $b\in{}^{0}K_{0}$ means that $b\blacktriangleright0\subseteq K_{0}$ and
$a\in{}^{1}K_{1}$ means that $a\blacktriangleleft1\subseteq K_{1}$, so by
Lemma \ref{Lem:23} we get%
\[
\left(  a-b\right)  ^{+}\blacktriangleleft1=b\blacktriangleright0\vee
a\blacktriangleleft1\subseteq K_{0}\vee K_{1}.
\]
In other words $\left(  a-b\right)  ^{+}\in{}^{1}K$, with the consequence
that
\[
\bigvee_{b\in{}^{0}K_{0}}\bigvee_{a\in{}^{1}K_{1}}\cone\theta\left(
a-b\right)  ^{+}\leq\bigvee_{^{1}K}\cone\theta\left(  c\right)  .
\]

Proposition \ref{Prop:5} provides the key step in the proof of the opposite
inequality. Fix some $a\in{}^{1}K_{1}$, and remember that $g\left(  \top
,K_{1}\right)  \geq\cone\theta\left(  a\right)  $. Consider any $b\in{}^{1}K$,
and remember that $g\left(  \bot,K_{0}\right)  \geq\bigvee
_{b\blacktriangleleft1}\coz\theta\left(  c\right)  $. Then see that
\begin{gather*}
g\left(  \bot,K_{0}\right)  \vee g\left(  \top,K_{1}\right)  \geq
\cone\theta\left(  a\right)  \vee g\left(  \bot,K_{0}\right)  \vee g\left(
\bot,K_{1}\right)  =\\
\cone\theta\left(  a\right)  \vee\bigvee_{^{0}K}\coz\theta\left(  c\right)
\geq\cone\theta\left(  a\right)  \vee\bigvee_{b\blacktriangleleft1}%
\coz\theta\left(  c\right)  \geq\cone\theta\left(  b\right)  .
\end{gather*}
Since $b$ was chosen arbitrarily, we get $g\left(  \bot,K_{0}\right)  \vee
g\left(  \top,K_{1}\right)  \geq\bigvee_{^{1}K}\cone\theta\left(  c\right)  $,
as desired.
\end{proof}

\begin{lemma}
$g$ preserves all joins.
\end{lemma}

\begin{proof}
Consider the join $\bigvee_{I}\left(  \varepsilon_{i},K_{i}\right)  =\left(
\varepsilon,\bigvee_{I}K_{i}\right)  $, and let $I_{0}\equiv\left\{  i\in
I:\varepsilon_{i}=\bot\right\}  $ ($I_{1}\equiv\left\{  i\in I:\varepsilon
_{i}=\top\right\}  $). Then this join can be parsed as
\[
\bigvee_{I}\left(  \varepsilon_{i},K_{i}\right)  =\bigvee_{I_{0}}\left(
\varepsilon_{i},K_{i}\right)  \vee\bigvee_{I_{1}}\left(  \varepsilon_{i}%
,K_{i}\right)  ,
\]
and, by Lemmas \ref{Lem:25} and \ref{Lem:26}, $g$ preserves the joins on the right.
\end{proof}

The proof of Theorem \ref{Thm:3} is complete.

\section{$\mathbf{W}$ is monoreflective in $\mathbf{AT}$\textbf{\label{Sec:3}}}

In this section we establish that $\mathbf{W}$ is monoreflective in
$\mathbf{AT}$, i.e., that every $\mathbf{AT}$-object is the domain of a
$\mathbf{W}$-universal arrow.

\subsection{Characterizing $\mathbf{W}$-objects}

\begin{proposition}
\label{Prop:6}The following are equivalent for a trunc $A$.

\begin{enumerate}
\item $A$ lies in $\mathbf{W}$, i.e., $A^{+}$ contains an element $a_{0}$ such
that $\overline{a}=a\wedge a_{0}$ for all $a\in A^{+}$.

\item There is some element $a_{0}\in\overline{A}$ for which $a_{0}%
\blacktriangleleft1=0$.

\item $\mathcal{M}A\in\mathbf{ipF}$, i.e., the designated point of
$\mathcal{M}A$ is isolated.

\item $\overline{A}$ contains a greatest element.
\end{enumerate}
\end{proposition}

\begin{proof}
(1) implies (2). Suppose $A^{+}$ contains an element $a_{0}$ such that
$\overline{a}=a\wedge a_{0}$ for all $a\in A^{+}$. Observe that $\left[
a_{0}\right]  =A$, for $\left[  a_{0}\right]  $ contains $\overline{a}=a\wedge
a_{0}$ for all $a\in A^{+}$, and $\left[  a_{0}\right]  $ has the closure
property of Lemma \ref{Lem:5}(2). Therefore, for $s<1$ in $\mathbb{F}$ we
have
\begin{gather*}
\left[  a_{0}\ominus s\right]  = \left[  s\left(  \frac{a_{0}}{s}%
\ominus1\right)  \right]  = \left[  s\left(  \frac{a_{0}}{s}-a_{0}\right)
^{+}\right]  = \left[  \left(  1-s\right)  a_{0}\right]  =\left[
a_{0}\right]  =A.
\end{gather*}
The second equality is is a consequence of the fact that, since truncation in
$A$ is given by meet with $a_{0}$, diminution is given by the rule
$a\ominus1=\left(  a-a_{0}\right)  ^{+}$, $a\in A^{+}$. The point is that
\[
a_{0}\blacktriangleleft1\equiv\bigvee_{0<s<1}\left[  a_{0}\ominus s\right]
^{\ast}=0,
\]
the bottom element of $\mathcal{K}A$.

(2) is equivalent to (3). According to the categorical equivalence between
$\mathbf{pF}$ and $\mathbf{fF}$ outlined in Proposition \ref{Prop:1},
$\mathcal{M}A$ is isolated iff $\mathcal{DM}A=\left(  \mathcal{K}A,F\right)  $
is fully filtered, i.e. iff the filter $F$ used to define $\mathcal{M}A$ from
$\mathcal{K}A$ is improper, meaning $0\in F$. But $F$ is generated by
truncation kernels of the form $a\blacktriangleleft1$, $a\in\overline{A}$.

(2) implies (4). Suppose that, for some $a_{0}\in\overline{A}$, we have
\[
0=a_{0}\blacktriangleleft1=\bigvee_{s<1}\left[  a_{0}\ominus s\right]  ^{\ast
}.
\]
It follows that $\left[  a_{0}\ominus s\right]  =A$ for all $s<1$. According
to Lemma 3.37 of \cite{Ball:2013},
\[
a_{0}\ominus s\wedge\left(  a-a_{0}\right)  ^{+}\ominus\left(  1-s\right)
\leq\left(  a\vee a_{0}\right)  \ominus1=0
\]
for all $a\in\overline{A}$. Therefore
\[
\left[  a_{o}\ominus s\right]  \wedge\left[  \left(  a-a_{0}\right)
^{+}\ominus\left(  1-s\right)  \right]  =0.
\]
Since $\left[  a_{0}\ominus s\right]  =A$, it follows that $\left[  \left(
a-a_{0}\right)  ^{+}\ominus\left(  1-s\right)  \right]  =0$. From Lemma
\ref{Lem:1} we then get
\[
\left[  \left(  a-a_{0}^{+}\right)  \right]  =\bigvee_{\mathbb{N}}\left[
\left(  a-a_{0}\right)  ^{+}\ominus\frac{1}{n}\right]  =0,
\]
with the result that $\left(  a-a_{0}\right)  ^{+}=0$, i.e., $a_{0}\geq a$.

(4) implies (1). Suppose that $\overline{A}$ contains greatest element $\overline{b}$. 
Then for any $a \in A^+$ we have $a \wedge \overline{b} \leq \overline{a}$ by axiom 
($\mathfrak{T}1$). But $\overline{a} \leq a$ by the same axiom, and 
$\overline{a} \leq \overline{b}$ by hypothesis, with the result that $\overline{a} \leq a \wedge \overline{b}$. In short, (1) holds.
\end{proof}

\begin{proposition}
Let $A$ be a trunc in $\mathbf{W}$, and let $a_{0}$ be the largest element of
$\overline{A}$. Then, in $\widehat{A}=\mu_{A}\left[  A\right]  $, we have%
\[
\widehat{a_{0}}\left(  U\right)  =\left\{
\begin{array}
[c]{lll}%
\left(  \top,\top\right)  & , & 0,1\in U\\
\left(  \top,\bot\right)  & , & 1\notin U\ni0\\
\left(  \bot,\top\right)  & , & 0\notin U\ni1\\
\left(  \bot,\bot\right)  & , & 0,1\notin U
\end{array}
\right.  .
\]

\end{proposition}

\begin{proof}
By definition (see Theorem \ref{Thm:4}) we have
\[
\widehat{a_{0}}\left(  r,\infty\right)  \equiv\left\{
\begin{array}
[c]{ll}%
\left(  \bot,a_{0}\blacktriangleright r\right)  , & r\geq0\\
\left(  \top,\top\right)  , & r<0
\end{array}
\right.  =\left\{
\begin{array}
[c]{cc}%
\left(  \bot,\left[  a_{0}\ominus r\right]  \right)  , & r\geq0\\
\left(  \top,\top\right)  , & r<0
\end{array}
\right.  ,\;
\]
This comes to $\left(  \bot,\left[  a_{0}\right]  \right)  $ when $r=0$, and
for $r>0$ we get
\begin{gather*}
\widehat{a_{0}}\left(  r,\infty\right)  = \left(  \bot,a_{0}%
\blacktriangleright r\right)  = \left(  \bot,\left[  a_{0}\ominus r\right]
\right)  = \left(  \bot,\left[  r\left(  \frac{a_{0}}{r}\ominus1\right)
\right]  \right)  =\\
\left(  \bot,\left[  r\left(  \frac{a_{0}}{r}-a_{0}\right)  \right] \right)  =
\left(  \bot,\left[  \left(  1-r\right)  ^{+}a_{0}\right]  \right)  = \left\{
\begin{array}
[c]{lll}%
\left(  \bot,\left[  \left(  1-r\right)  ^{+}a_{0}\right]  \right)  & , &
0<r<1\\
\left(  \bot,\bot\right)  & , & r\geq1
\end{array}
\right.  =\\
\left\{
\begin{array}
[c]{ccc}%
\left(  \bot,\left[  a_{0}\right]  \right)  & , & 0<r<1\\
\left(  \bot,\bot\right)  & , & r\geq1
\end{array}
\right.  .
\end{gather*}
The last equality results from the fact that truncation kernels are closed
under scalar multiplication.

We claim that $\left[  a_{0}\right]  =\top$. This follows directly from two
facts: first, $\left[  a_{0}\right]  \supseteq\overline{A}$ since $a_{0}$ is
the greatest element of $\overline{A}$, and second, $\left[  a_{0}\right]  $
satisfies property (2) of Lemma \ref{Lem:5}. Thus we can summarize the
situation as follows.%
\[
\widehat{a_{0}}\left(  r,\infty\right)  =\left\{
\begin{array}
[c]{lll}%
\left(  \top,\top\right)   & , & r<0\\
\left(  \bot,\top\right)   & , & 0\leq r<1\\
\left(  \bot,\bot\right)   & , & r\geq1
\end{array}
\right.
\]
In light of the fact that $\widehat{a_{0}}\left(  -\infty,r\right)
=\bigvee_{s<r}\widehat{a_{0}}\left(  s,\infty\right)  ^{\ast}$, this
information supports the inference that
\[
\widehat{a_{0}}\left(  -\infty,r\right)  =\left\{
\begin{array}
[c]{lll}%
\left(  \top,\top\right)   & , & r>1\\
\left(  \top,\bot\right)   & , & 0<r\leq1\\
\left(  \bot,\bot\right)   & , & r\leq0
\end{array}
\right.  .
\]
The proposition itself is a consequence of the last two displayed equations.
\end{proof}

Let $A$ be a trunc with spectrum $\mathcal{M}A=2_{F}\mathcal{K}A$, and let
$\nu_{M}$ designate the insertion $2_{F}\mathcal{K}A\rightarrow2KA$. Applying
the $\mathcal{R}_{\mathbf{pF}}$ functor to $\nu_{M}$ and composing the result
with $\mu_{A}$ provides a trunc injection $A\rightarrow B\equiv\mathcal{R}%
_{\mathbf{pF}}\left(  2\mathcal{K}A\right)  $ which we denote by
\[
\omega_{A}\equiv\mathcal{R}_{\mathbf{pF}}\nu_{M}\circ\mu_{A}.
\]
We abuse the notation to the extent of using $\widehat{a}$ to denote
$\omega_{A}\left(  a\right)  $ and $\widehat{A}$ to denote $\left\{
\widehat{a}:a\in A\right\}  $, trusting the reader to supply the appropriate
meaning from context.

Now $\mathcal{M}B$, being isomorphic to $2\mathcal{K}A$, is an isolated point
frame, so that $\overline{B}$ has a greatest element $b_{0}$ by Proposition
\ref{Prop:6}. We define
\[
\omega A\equiv\left\langle \widehat{A},b_{0}\right\rangle ,
\]
the subtrunc of $B$ generated by $\widehat{A}\cup\left\{  b_{0}\right\}  $.

\begin{theorem}
$\mathbf{W}$ is monoreflective in $\mathbf{AT}$, and $\omega_{A}:A\rightarrow
\omega A$ is a reflector for $A\in\mathbf{AT}$.
\end{theorem}

\begin{proof}
Consider the $\mathbf{AT}$-morphism\textbf{\ }$\theta:A\rightarrow
C\in\mathbf{W}$, so that $\mathcal{M}C=2\mathcal{K}C$ is isolated. Let
$g:2_{F}\mathcal{K}A\rightarrow2\mathcal{K}C$ be the unique pointed frame map
such that $\mathcal{R}_{\mathbf{pF}}g\circ\mu_{A}=\mu_{C}\circ\theta$. Now $g$
extends uniquely over $\nu_{M}$ since the latter is the free isolated point
frame over $2_{F}\mathcal{K}A$, thus providing a unique pointed frame map
$g^{\prime}:2\mathcal{K}A\rightarrow2\mathcal{KC}$ such that $g^{\prime}%
\circ\nu_{M}=g$. Then, as the reader may easily check, $\mathcal{R}%
_{\mathbf{pF}}g^{\prime}:R_{0}\left(  2KA\right)  \longrightarrow R_{0}\left(
2KB\right)  $,
\end{proof}

\end{document}